\documentclass{amsart}
\usepackage{graphicx, tikz}
\usepackage{epsfig}
\usepackage{amssymb, amsfonts, latexsym, mathrsfs, amsmath, amsthm, graphicx, bbm, amstext}
\usepackage{txfonts,bbding,wasysym,pifont,stmaryrd,tipa}
\usepackage[all,cmtip]{xy}
\usepackage{float, mathtools, enumitem}
\usepackage[utf8]{inputenc}
\usepackage{hyperref}
\usepackage{cleveref}
\crefname{section}{§}{§§}
\Crefname{section}{§}{§§}
\usetikzlibrary{arrows,positioning,shapes.geometric,matrix}

\newtheorem{theorem}{Theorem}[section]
\newtheorem{lemma}[theorem]{Lemma}
\newtheorem{proposition}[theorem]{Proposition}
\newtheorem{corollary}[theorem]{Corollary}
\newtheorem{conj}[theorem]{Conjecture}
\newtheorem{ass}[theorem]{Assumption}
\newtheorem{que}[theorem]{Question}

\theoremstyle{definition}
\newtheorem{definition}[theorem]{Definition}
\newtheorem{example}[theorem]{Example}

\theoremstyle{remark}
\newtheorem{remark}[theorem]{Remark}
\newtheorem{notation}[theorem]{Notation}
\numberwithin{equation}{section}

\newcommand{\Sym}{\text{Sym}}

\newcommand{\Supp}{\text{Supp}}
\newcommand{\Aut}{\text{Aut}}

\DeclareMathAlphabet{\mathpzc}{OT1}{pzc}{m}{it}

\newcommand{\bb}[1]{\mathbb{#1}}
\newcommand{\tx}[1]{\text{#1}}
\newcommand{\cal}[1]{\mathcal{#1}}

\newcommand{\surj}{\twoheadrightarrow} 
 
\DeclareMathOperator{\sheafHom}{\mathscr{H}\text{\kern -3pt {\calligra\large om}}\,}

\begin{document}

\title{On the Inertia Conjecture and its generalizations}
\author{Soumyadip Das}
\address{Statistics and Mathematics Unit,
Indian Statistical Institute, Bangalore Centre, Bangalore 560059.}
\email{ soumyadip\_rs@isibang.ac.in}

\subjclass[2010]{11S15, 14H30 (Primary) 13B05, 14G17 (Secondary)}

\keywords{Galois Cover, Inertia Conjecture, Ramification}

\begin{abstract}
Studying two point branched Galois covers of the projective line we prove the Inertia Conjecture for the Alternating groups $A_{p+1}$, $A_{p+3}$, $A_{p+4}$ for any odd prime $p \equiv 2 \pmod{3}$ and for the group $A_{p+5}$ when additionally $4 \nmid (p+1)$ and $p \geq 17$. We obtain a generalization of a patching result by Raynaud which reduces these proofs to showing the realizations of the \'{e}tale Galois covers of the affine line with a fewer candidates for the inertia groups above $\infty$. We also pose a general question motivated by the Inertia Conjecture and obtain some affirmative results. A special case of this question, which we call the Generalized Purely Wild Inertia Conjecture, is shown to be true for the groups for which the purely wild part of the Inertia Conjecture is already established. In particular, we show that if this generalized conjecture is true for the groups $G_1$ and $G_2$ which do not have a common quotient, then the conjecture is also true for the product $G_1 \times G_2$.
\end{abstract}

\maketitle

\section{Introduction}
Let $k$ be an algebraically closed field of characteristic $p>0$ and $U$ be a smooth connected affine $k$-curve. The full structure of the \'{e}tale fundamental group $\pi_1(U)$ is not known in general. To this end one interesting problem is to understand the set $\pi_A(U)$ of isomorphism classes of the finite (continuous) group quotients of $\pi_1(U)$, or equivalently the finite groups that occur as the Galois groups of Galois \'{e}tale connected covers of $U$. In 1957 Abhyankar conjectured (known as Abhyankar's Conjecture on the affine curves, \cite{6}, \cite[Conjecture 3.2]{survey_paper}; now a Theorem due to Serre, Raynaud and Harbater) on which groups occur in the set $\pi_A(U)$. In particular, this shows that $\pi_A(\bb{A}^1)$ is the set of isomorphism classes of the quasi $p$-groups, i.e. the groups $G$ which are generated by its Sylow $p$-subgroups. Given $G \in \pi_A(U)$, the next natural problem is to study the inertia groups that occur over the points in $X - U$. By \cite[Chapter IV, Corollary 4]{Serre_loc}, any inertia group must be of the form $P \rtimes \bb{Z}/m$ for a $p$-group $P$ and a coprime to $p$ integer $m$. Studying the branched covers of $\bb{P}^1$ using explicit equations Abhyankar proposed the following conjecture, known as the Inertia Conjecture (IC).

\begin{conj}\label{IC}
(IC, \cite[Section 16]{Abh_01}) Let $G$ be a finite quasi $p$-group. Let $I$ be a subgroup of $G$ which is an extension of a $p$-group $P$ by a cyclic group of order prime-to-$p$. Then there is a connected $G$-Galois cover of $\bb{P}^1$ \'{e}tale away from $\infty$ such that $I$ occurs as an inertia group at a point over $\infty$ if and only if the conjugates of $P$ in $G$ generate $G$.
\end{conj}

A special case of the above conjecture when $I$ is a $p$-group is known as the Purely Wild Inertia Conjecture (PWIC for short).

\begin{conj}\label{PWIC}
(PWIC) Let $G$ be a finite quasi $p$-group. A $p$-subgroup $P$ of $G$ occurs as the inertia group at a point above $\infty$ in a connected $G$-Galois cover of $\mathbb{P}^1$ branched only at $\infty$ if and only if the conjugates of $P$ generate $G$.
\end{conj}

If a $G$-Galois cover exists as in Conjecture \ref{IC}, then it is well known that the conjugates of the $p$-subgroup $P$ in $G$ generate $G$ and this follows from the fact that the tame fundamental group of the affine line is trivial. The other direction of the conjecture is proved to be true only in a few cases (see \cite{BP}, \cite{12}, \cite{8}, \cite{15}, \cite{DK} and \cite{survey_paper} for more details) and even the PWIC remains wide open at this moment. Using a Formal Patching technique Harbater has shown that the PWIC is true when $P$ is a Sylow $p$-subgroup of $G$. One particular class of groups of interest are the Alternating groups. It is known that the IC is true for the group $A_p$ when $p \geq 5$ (\cite[Theorem 1.2]{BP}) and for $A_{p+2}$ when $p \equiv 2 \pmod{3}$ is an odd prime (\cite[Theorem 1.2]{8}). Now onward we assume that $p$ is an odd prime.

In this paper we prove the following result.

\begin{theorem}\label{intro_IC}
When $p \equiv 2 \pmod{3}$ is an odd prime, the IC is true for the groups $A_{p+1}$, $A_{p+3}$ and $A_{p+4}$. When $p \equiv 2 \pmod{3}$, $4 \nmid (p+1)$ and $p \geq 17$ the IC is true for $A_{p+5}$.
\end{theorem}

The result for $A_{p+1}$ is of special interest since till now there was no example of an $A_{p+1}$-Galois \'{e}tale cover of the affine line such that the tame part of the inertia group at a point above $\infty$ is non-trivial. The strategy of the proofs is to first identify the inertia groups that are needed to be realized over $\infty$ for the corresponding Galois \'{e}tale covers of the affine line. Using Abhyankar's Lemma (\cite[XIII, Proposition 5.2]{10}) and Lemma \ref{lem_Alt_FP} we reduce this list of potential inertia groups. Then we construct two points branched covers of $\bb{P}^1$ given by some explicit affine equations (Section \cref{sec_explicit_eq}) to resolve these cases.

Now let $U$ be as before and let $U \subset X$ be its smooth projective completion. Let $\phi\colon Y \to X$ be a connected Galois cover with group $G$ which is \'{e}tale away from $B \coloneqq X - U$ and let $I_x$ occurs as an inertia group above $x \in B$. Set $H \coloneqq \langle p(I_x)^G | x \in B \rangle$ where $p(I_x)$ denotes the Sylow $p$-subgroup of $I_x$. Then the cover $\phi$ factors through a connected $G/H$-Galois cover $\psi \colon Y' \to X$ \'{e}tale away from $B$ such that $I_x/I_x \cap H$ occurs as an inertia group above $x \in B$. One can again do the same for the newly obtained cover $\psi$. Motivated by this we pose Question \ref{que_most_gen}, which we call $Q[r,X,B,G]$. Here we state two special cases of this question for which we have obtained several positive results.

\begin{que}\label{intro_gen_que}
(See Question \ref{que_most_gen}; $Q[r,X,B,G]$) Let $r \geq 1$ be an integer, $X$ be a smooth projective connected $k$-curve. Let $B \coloneqq \{x_1,\cdots, x_r\}$ be a set of closed points in $X$. Let $G$ be a finite group, $P_1$, $\cdots$, $P_r$ be (possibly trivial except for $P_1$) $p$-subgroups of $G$. Set $H \coloneqq \langle P_i^G | 1 \leq i \leq r \rangle$. Assume that either of the following holds.
\begin{enumerate}
\item $G = \langle P_1^G, \cdots, P_r^G \rangle$;
\item $H = \langle P_1^H, \cdots, P_r^H \rangle$.
\end{enumerate}
For $1 \leq i \leq r$ let $I_i$ be a subgroup of $G$ which is an extension of the $p$-group $P_i$ by a cyclic group of order prime-to-$p$ such that there is a connected $G/H$-Galois cover $\psi$ of $X$ \'{e}tale away from $B$ and $I_i/I_i \cap H$ occurs as an inertia group above $x_i$, $1 \leq i \leq r$. Does there exist a connected $G$-Galois cover $\phi$ of $X$ \'{e}tale away from $B$ dominating the cover $\psi$ such that $I_i$ occurs as an inertia group above the point $x_i \in B$ ?
\end{que}

Using Formal Patching and studying the branched covers of curves given by explicit equations we obtain some affirmative answers to the above question. We have the following result towards $Q[1,X,\{*\},S_d]$.

\begin{theorem}\label{intro_thm_Sym}
(Corollary \ref{cor_gen_que_sym}) Let $p$ be a prime $\geq 5$ and $X$ be any smooth projective $k$-curve of genus $\geq 1$. Then for $r = 1$, Question \ref{intro_gen_que} has an affirmative answer for the group $S_p$ and when $p \equiv 2 \pmod{3}$ for the groups $S_{p+1}$, $S_{p+2}$, $S_{p+3}$, $S_{p+4}$.
\end{theorem}

The above result is a corollary to Proposition \ref{prop_part_most_gen} which suggests that the answer to Question \ref{intro_gen_que} which is related to the understanding of the structure of the group $\pi_1(X - B)$ has a close connection with the IC, the tame fundamental group $\pi_1^t(X - B)$ and the group theoretic behavior of the Galois groups. Using the existence of the HKG covers we show that (Theorem \ref{thm_semidirect}) $Q[2,\bb{P}^1,\{0,\infty\},P \rtimes \bb{Z}/n]$ has an affirmative answer. For $X = \bb{P}^1$ and $r=2$ we have the following result towards $Q[2,\bb{P}^1,\{0,\infty\},S_d]$.

\begin{theorem}\label{intro_thm_general}
(Theorem \ref{thm_S_p_both_Sylow_non-trivial}, \ref{thm_S_p+1_both_Sylow_non-trivial}) For $r=2$, $X = \bb{P}^1$, Question \ref{intro_gen_que} has an affirmative answer when $p \equiv 2 \pmod{3}$, $G = S_p$ or $S_{p+1}$ and both $P_1$ and $P_2$ are non-trivial.
\end{theorem}

We also have the similar results (Theorem \ref{thm_S_P+2_both_Sylow_non-trivial}, \ref{thm_S_P+3_both_Sylow_non-trivial}) for the groups $S_{p+2}$ and $S_{p+3}$ with more restrictions on $p$.

One special case of Question \ref{intro_gen_que} is when $X = \bb{P}^1$ which we pose as the Generalized Inertia Conjecture (Conjecture \ref{conj_proj_gen}). Even more specializing to the case when the inertia groups are the $p$-groups we pose the Generalized Purely Wild Inertia Conjecture (GPWIC, Conjecture \ref{conj_GPWIC}). We see that for the groups for which the PWIC is already known to be true, the GPWIC is also true. Namely, we prove the following result.

\begin{theorem}
(Corollary \ref{cor_all}) Let $G$ be a quasi $p$-group, $P_1$, $\cdots$, $P_r$ are $p$-subgroups of $G$ for some $r \geq 1$ such that $G =\langle P_1^G , \cdots, P_r^G \rangle$. Let $B \coloneqq \{x_1,\cdots, x_r\}$ be a set of closed points in $\bb{P}^1$. There is a connected $G$-Galois cover of $\bb{P}^1$ \'{e}tale away from $B$ such that $P_i$ occurs as an inertia group above $x_i$ where $G$ is one of the following groups.
\begin{enumerate}
\item $G$ is a $p$-group;
\item $G$ has order strictly divisible by $p$;
\item $G = G_1 \times \cdots \times G_u$ where each $G_i$ is either a simple Alternating group of degree coprime to $p$ or a $p$-group or a simple non-abelian group of order strictly divisible by $p$.
\end{enumerate}
\end{theorem}

We also show that the GPWIC holds for certain product of groups if it holds for individual groups. This generalizes \cite[Corollary 4.6]{15}.

\begin{theorem}
(Theorem \ref{thm_pdt_arbit_groups}) Let $G_1$ and $G_2$ be two finite quasi $p$-groups such that they have no non-trivial quotient in common. If the GPWIC is true for the groups $G_1$ and $G_2$, then the GPWIC is also true for $G_1 \times G_2$.
\end{theorem}

The structure of the article is as follows. In Section \cref{sec_explicit_eq} we introduce covers given by some explicit general affine equations and study their ramification behaviour. In particular, we obtain some general Symmetric and Alternating group covers of $\bb{P}^1$ \'{e}tale away from $\{0, \infty\}$ with a well described tame ramification over $0$ and a wild ramification over $\infty$ whose $p$-part is generated by a $p$-cycle. In Section \cref{sec_fp} we introduce and recall some results on constructing covers using Formal Patching technique. In Section \cref{sec_alt_IC} we see some new evidence (Theorem \ref{intro_IC}) for the IC for the Alternating group covers. In Section \cref{sec_questions} we introduce a general questions and some conjectures generalizing the Inertia Conjecture. In Sections \cref{sec_GPWIC} and \cref{sec_GPWIC_weak} we see some results in the support of the GPWIC. Finally Section \cref{sec_que_evidence} contains some general results toward Question \ref{intro_gen_que} and the particular cases of Symmetric and other groups of small degree.

\subsection*{Acknowledgements} I would like to thank my advisor, Manish Kumar, for his suggestions and comments which helped improving the structure and the style of the manuscript.

\section{Notation and Conventions}\label{sec_notation}
We fix the following notation throughout this article.

\begin{enumerate}
\item $p$ denotes an odd prime, $k$ denotes an algebraically closed field of characteristic $p$.
\item All the $k$-curves considered will be smooth connected curves, unless otherwise specified.
\item All the Alternating and Symmetric groups considered will be of degree $\geq 5$.
\item For a finite group $G$, $p(G)$ denotes the subgroup of $G$ generated by all the Sylow $p$-subgroups of $G$.
\end{enumerate}

\begin{definition}
For a finite group $G$ and a subgroup $I$ we say that the pair $(G,I)$ is \textit{realizable} if there exists a connected $G$-Galois cover of $\bb{P}^1$ branched only at $\infty$ such that $I$ occurs as an inertia group above $\infty$.
\end{definition}

In the above definition $G$ is necessarily a \textit{quasi} $p$-\textit{group} i.e. $G = p(G)$.

\begin{definition}\label{def_Kummer_pullback}
Let $n$ be coprime to $p$. \textit{The} $[n]$-\textit{Kummer cover} is the unique connected $\bb{Z}/n$-Galois cover $\psi \colon Z \cong \bb{P}^1 \to \bb{P}^1$ \'{e}tale away from $\{0,\infty\}$ over which the cover is totally ramified. Let $\phi \colon Y \to \bb{P}^1$ be a connected $G$-Galois cover. Let $W$ be a dominant component in the normalization of $Y \times_{\bb{P}^1} Z$. We say that \textit{the cover} $W \to Z$ \textit{is obtained by a pullback by the} $[n]$-\textit{Kummer cover}.
\end{definition}

\section{Constructing Covers by Explicit Equations}\label{sec_explicit_eq}
In this section we construct covers of $\bb{P}^1$ given by some explicit affine equations and study the Galois closure of these covers. Our primary interest are the cases where the Galois groups are Alternating or Symmetric groups of degree $d \geq p$. The covers will be constructed as the Galois closure of a finite degree-$d$ cover $\bb{P}^1 \to \bb{P}^1$ \'{e}tale away from $\{0,\infty\}$. Any finite cover $\bb{P}^1_y \to \bb{P}^1_x$ is given by an affine equation of the form $x f(y) - g(y) =0$ for some polynomials $f(y)$ and $g(y)$ in $k[y]$ such that they do not have any common zero. We impose conditions on these polynomials so that the resulting cover has desired inertia and Galois groups. The main results of this section will be implemented in the later parts and in particular to prove the Inertia Conjecture for certain Alternating groups (cf. Section \cref{sec_alt_IC}). Before introducing the explicit construction of covers we review and obtain some results.

Let $d \geq p$, and $\tau$ be a $p$-cycle in $S_d$. By \cite[Proposition 2.1]{DK}, for any transitive quasi $p$-subgroup $G$ of $S_d$, containing $\tau$, we have 

\begin{equation}\label{eq_theta}
N_G(\langle \tau \rangle) = (\langle \tau, \theta \rangle \times J) \cap G
\end{equation}
where $J$ is the Symmetric group on the set $\{1,\cdots, d\} \setminus \Supp(\tau)$ and $\theta \in \Sym(\Supp(\tau)) \cap N_{S_d}(\langle \tau \rangle)$ an element of order $p-1$ such that the conjugation by $\theta$ is a generator of $\Aut(\langle \tau \rangle)$. Arguing as in the proof of \cite[Proposition 4.16]{survey_paper} one obtains the following result.

\begin{lemma}\label{lem_fibres_I_action}
Let $d \geq p$, and $G$ be a transitive subgroup of $S_d$. Let $\phi \colon Z \to X$ be a $G$-Galois cover of smooth connected projective $k$-curves, and $x \in X$ be a closed point. Let $I$ occurs as an inertia group over $x$. Consider the subgroup $S_{d-1}$ of $S_d$ fixing the element $1$. Set $H \coloneqq G \cap S_{d-1}$. Let $\psi \colon Y \to X$ be the connected degree $d$ cover via which $\phi$ factors. Then $\psi^{-1}(x)$ in $Y$ is in a bijective correspondence with the set of orbits of the action of $I$ on $\{1,\cdots,d\}$. Moreover, for a point $y \in \psi^{-1}(x)$, the ramification index of $y$ over $x$ is given by the length of the corresponding orbit.
\end{lemma}

\begin{lemma}\label{lem_transitive_grp__int_grp_structure_from_fibres}
Under the hypothesis of Lemma \ref{lem_fibres_I_action}, we have the following.
\begin{enumerate}
\item If $\psi^{-1}(x)$ consists of $s$ points with the ramification indices $n_1$, $\cdots$, $n_s$ with each $n_i$ coprime to $p$, $\sum_{i=1}^s n_i =d$, then $\phi$ is tamely ramified over $0$. If $\gamma$ is a generator of $I$, then the disjoint cycle decomposition of $\gamma$ in $S_d$ consists of $s$ cycles of length $n_1$, $\cdots$, $n_s$.\label{tame}
\item If $d=p$ and $\psi^{-1}(x)$ consists of a unique point with the ramification index $p$, then $I$ is of the form $I = \langle \tau \rangle \rtimes \langle \theta^i \rangle$ for a $p$-cycle $\tau$ and some $1 \leq i \leq p-1$. If $d \geq p+1$ and $\psi^{-1}(x)$ consists of $r+1$ points with the ramification indices $p$, $m_1$, $\cdots$, $m_r$ with all $m_l$ coprime to $p$ and $\sum_{l=1}^r m_l = d-p$, then $I$ is of the form $I = \langle \tau \rangle \rtimes \langle \theta^i \omega \rangle$ for a $p$-cycle $\tau$, $1 \leq i \leq p-1$, and an $\omega \in \tx{Sym}(\{p+1,\cdots,d\})$ having a disjoint cycle decomposition consisting of $r$ cycles of length $m_1$, $\cdots$, $m_r$. Here $\theta$ is as in Equation (\ref{eq_theta}).\label{wild}
\end{enumerate}
\end{lemma}

\begin{proof}
We prove (\ref{wild}). The result (\ref{tame}) is well known and can be proved using a similar argument. Let $z \in \phi^{-1}(x) \subset Z$ having image $y \in Y$ such that $I$ is the inertia group at $z$. We first claim that $p^2 \nmid |I|$. Assume on the contrary. Let $\tau \in p(I)\cap H$ be an element of order $p$. For any point $y' \in \psi^{-1}(x)$, $y' \neq y$, the ramification index at $y'$ over $x$ is coprime to $p$, and the inertia groups are conjugate to each other. So $g^{-1}\tau g \in H$ for all $g \in G$. But since $G$ acts transitively on $\{1,\cdots,d\}$, there is a $g \in G$ such that $g^{-1}\tau g$ does not fix $1$, a contradiction.

Since $|I|$ is divisible by $p$, $p(I)$ is a $p$-cyclic group generated by an element $\tau = \tau_1 \cdots \tau_a$ of order $p$, where $\tau_i$ are disjoint $p$-cycles in $S_d$. By \cite[Proposition 2.1]{DK}, $I = \langle \tau,\theta^i,\sigma \rangle \times \langle \omega \rangle$ for some $1 \leq i \leq p-1$, $\omega \in \tx{Sym}(\{1,\cdots,d\} - \tx{Supp}(\tau))$ and $\sigma \in \tx{Sym}(\tx{Supp}(\tau))$ of order prime-to-$p$. By Lemma \ref{lem_fibres_I_action}, the fibre $\psi^{-1}(x)$ consists of points with the ramification indices $a_1 p$, $\cdots$, $a_t p$, $u_1$, $\cdots$, $u_{t'}$, where $a_{\nu}$ and $u_{\eta}$ are coprime to $p$.

So if $d=p$ and $\psi^{-1}(x)$ consists of a unique point with the ramification index $p$, then $\tau$ must be a $p$-cycle and $I$ is of the form $I = \langle \tau \rangle \rtimes \langle \theta^i \rangle$ for some $1 \leq i \leq p-1$. In the second case, $\tau$ is again a $p$-cycle, $I = \langle \tau \rangle \rtimes \langle \theta^i \omega \rangle$ for some $1 \leq i \leq p-1$, and the disjoint cycle decomposition of $\omega$ in $\tx{Sym}(\{p+1,\cdots,d\})$ consists of $r$ cycles length $m_1$, $\cdots$, $m_r$.
\end{proof}

Using the above results and a technique used in \cite[Proposition 1.3]{9} (when $p$ strictly divides the order of $G$), we obtain the following result for a certain type of two point branched Galois cover of the projective line $\bb{P}^1$.

\begin{proposition}\label{prop_ram_grp}
Let $p$ be a prime, $d \geq p$. Let $G$ be a transitive subgroup of $S_d$. Let $\phi \colon Z \to \bb{P}^1$ be a $G$-Galois cover of smooth projective connected $k$-curves. Consider the degree-$d$ cover $\psi \colon Y \coloneqq Z/(G\cap S_{d-1}) \to \bb{P}^1$ of smooth projective connected $k$-curves where $S_{d-1}$ is the subgroup of elements in $S_d$ fixing $1$. Assume that the following hold. 
\begin{enumerate}[label=(\roman*)]
\item There are exactly $s$ points in $\psi^{-1}(0)$ with ramification indices $n_1$, $\cdots$, $n_s$ such that each $n_i$ is coprime to $p$ and $\sum_{i=1}^s n_i = d$;
\item when $d=p$, there is a unique point in $Y$ lying over $\infty$ with ramification index $p$. When $d>p$, there are exactly $r+1$ points in $\psi^{-1}(\infty)$ with ramification indices $p$, $m_1$, $\cdots$, $m_r$ such that each $m_l$ is coprime to $p$ and $\sum_{l=1}^r m_l = d-p$.
\end{enumerate}
Then $\phi$ is tamely ramified over $0$ and $I = \langle (1, \cdots, p) \rangle \rtimes \langle \theta^i \omega \rangle$ occurs as an inertia group over $\infty$ for some $1 \leq i \leq p-1$ ($\theta$, $\omega$ are as in Equation (\ref{eq_theta})). If $\gamma$ is a generator of an inertia group over $0$, the disjoint cycle decomposition of $\gamma$ in $S_d$ consists of $s$ cycles of length $n_1$, $\cdots$, $n_s$. If $d=p$, $\omega$ is the trivial permutation and if $d \geq p+1$, the disjoint cycle decomposition of $\omega$ in $\tx{Sym}(\{p+1,\cdots,d\})$ consists of $r$ cycles length $m_1$, $\cdots$, $m_r$.

Moreover, if the cover $\phi$ is \'{e}tale away from $\{0,\infty \}$ and $g(Y)$ is the genus of $Y$, the upper jump for any local extension above $\infty$ is $\frac{2g(Y)+s+r-1}{p-1}$ and $\tx{ord}(\theta^i) = \frac{p-1}{(p-1,2g(Y)+s+r-1)}$.
\end{proposition}

\begin{proof}
The structure of the inertia groups are the consequence of Lemma \ref{lem_transitive_grp__int_grp_structure_from_fibres}. Suppose that the cover $\phi$ is \'{e}tale away from $\{0,\infty\}$. We use the Riemann Hurwitz formula for the two Galois covers $\phi$ and $Z \to Y$ to obtain the upper jump $\frac{2g(Y)+s+r-1}{p-1}$ at $\infty$. By \cite[Lemma 2.6, Equation (2.2)]{DK}, $\tx{ord}(\theta^i) = \frac{p-1}{(p-1,2g(Y)+r+s-1)}$.
\end{proof}

Now we proceed towards construction of covers. We will consider the following assumption.

\begin{ass}\label{ass_num_deg_p}
Let $s \geq 2$ be an integer. Let $n_1$, $\cdots$, $n_s$ be coprime to $p$ such that $\Sigma_{i=1}^s n_i = p$. Let $\alpha_1 = 0$ and assume that there exist non-zero distinct elements $\alpha_2$, $\cdots$, $\alpha_s$ in $k$ so that the polynomial
\begin{equation*}
\Sigma_{i=1}^s n_i \Pi_{j \neq i, 1 \leq j \leq s} (y - \alpha_j) \in k[y]
\end{equation*}
is a non-zero constant in $k$. In terms of the coefficients of this polynomial, the assumption is equivalent to the existence of non-zero distinct $\alpha_i$'s, $2 \leq i \leq s$, such that for each $1 \leq \nu \leq s-2$, $\Sigma_{2 \leq i_1 < \cdots < i_{s-\nu-1} \leq s} (\Sigma_{i \in \{i_1, \cdots, i_{s-\nu-1}\}} n_i) \alpha_{i_1} \cdots \alpha_{i_{s-\nu-1}} = 0$.
\end{ass}

\begin{remark}\label{rmk_choice}
Note that Assumption \ref{ass_num_deg_p} is satisfied when $s=2$ or $s=3$. It is immediate when $s=2$. For $s=3$ the assumption holds for the pair $(\alpha_2,\alpha_3) = (1, - \frac{n_2}{n_3})$.
\end{remark}

Now we construct some degree-$p$ covers of $\bb{P}^1$ which will be \'{e}tale away from $\{0,\infty\}$.

\begin{proposition}\label{prop_deg_p_cover}
Let $p \geq 5$ be a prime and $s \geq 2$ be an integer. Let $n_1$, $\cdots$, $n_s$ be coprime to $p$ such that $\sum_{i=1}^s n_i = p$. Let $\alpha_1$, $\cdots$, $\alpha_s$ be distinct elements in $k$. Let $\psi \colon Y \to \bb{P}^1$ be the degree-$p$ cover given by the affine equation $\bar{f}(x,y)=0$ where
\begin{equation}\label{eq_explicit_deg_p}
\bar{f}(x,y) \coloneqq \Pi_{i=1}^s (y-\alpha_i)^{n_i} - x.
\end{equation}
Let $\phi \colon Z \rightarrow \mathbb{P}^1$ be the Galois closure of $\psi$ with group $G$. 
Then the following hold.
\begin{enumerate}
\item $G$ is a primitive subgroup of $S_p$;\label{p:1}
\item $\phi$ is tamely ramified with cyclic inertia group of order $\tx{l.c.m.} \{n_1, \cdots, n_s\}$ over $0$. If $\gamma$ is one of its generators, then $\gamma$ has a disjoint cycle decomposition in $S_p$ with cycle lengths $n_1$, $\cdots$, $n_s$;\label{p:2}
\item over $\infty$, the inertia group is of the form $I = \langle (1, \cdots, p) \rangle \rtimes \langle \theta^i \rangle$ for some $1 \leq i \leq p-1$ and where $\theta$ is as in Equation (\ref{eq_theta}).\label{p:3}
\end{enumerate}
Additionally, if $\alpha_i$'s satisfy Assumption \ref{ass_num_deg_p}, then the cover $\phi$ is \'{e}tale away from $\{0,\infty\}$. Also $\tx{ord}(\theta^i) = \frac{p-1}{(p-1,s-1)}$, $|I| = \frac{p(p-1)}{(p-1,s-1)}$, and the upper jump for any $I$-Galois local extension over $\infty$ is given by $\frac{s-1}{p-1}$. Moreover, if there is a positive integer $j$ such that $\gamma^j$ is a non-trivial cycle fixing $\geq 3$ points in $\{1,\cdots,p\}$ or if $p \neq 11$, $23$, $p \not\in \{\frac{q^n -1}{q-1}| q \tx{ prime power }, n \geq 2\}$, and $\gamma$ is not a conjugate of $\theta^i$ for any $1 \leq i \leq p-1$,
\[
    G = 
\begin{cases}
    A_p,& \tx{if } \gamma \tx{ is an even permutation}\\
    S_p,              & \tx{if } \gamma \tx{ is an odd permutation.}
\end{cases}
\]
\end{proposition}

\begin{proof}
The polynomial $\bar{f}(x,y)$ is linear and monic in $x$. So it is irreducible in $k[y][x]$ and hence in $k(x)[y]$. So $G$ is a transitive subgroup of $S_p$ and hence it is a primitive subgroup of $S_p$. This proves (\ref{p:1}).

From the equation $\bar{f}(x,y)=0$ it follows that $v_{(y-\alpha_i)}(x) = n_i$ for $1 \leq i \leq s$ and $v_{(y^{-1})}(x^{-1}) = p$. Since $\Sigma_i n_i = p$, we see that the fibre $\psi^{-1}(0)$ consists of $s$ points in $Y$ and the ramification index at the point $(y=\alpha_i)$ is given by $n_i$. Also there is a unique point in $Y$ lying above $\infty$ at which the ramification index is $p$. Then (\ref{p:2}) and (\ref{p:3}) follow from Proposition \ref{prop_ram_grp}.

Now suppose that $\alpha_i$'s satisfy Assumption \ref{ass_num_deg_p}. The $y$-derivative of $\bar{f}(x,y)$ is given by
$$\bar{f}_y(x,y) = \Pi_{i=1}^s (y - \alpha_i)^{n_i-1} (\Sigma_{i=1}^s n_i \Pi_{j \neq i, 1 \leq j \leq s} (y - \alpha_j)).$$
Let $(a,b)$ be a common zero of $\bar{f}$ and $\bar{f}_y$. Then $a = 0$ if $n_i > 1$ for some $i$ and there is no common zero if $n_i = 1$ for all $1 \leq i \leq s$. So the cover $\psi$, and hence $\phi$ is \'{e}tale away from $\{0,\infty\}$. By Proposition \ref{prop_ram_grp}, the upper jump is $\frac{s-1}{p-1}$, $\tx{ord}(\theta^i) = \frac{p-1}{(p-1,s-1)}$. So $|I| = \frac{p(p-1)}{(p-1,s-1)}$.

Since $G$ is a primitive subgroup of $S_p$ containing a $p$-cycle, under the additional hypothesis on $p$ and $\gamma$, $G$ contains $A_p$ by \cite[Theorem 1.2]{Jones}. So if $\gamma$ is an odd permutation, $G = S_p$. Now let $\gamma$ be an even permutation and assume that $G = S_p$. Then the connected $\bb{Z}/2$-Galois cover $Z/A_p \to \bb{P}^1$ is \'{e}tale away from $\infty$ and is tamely ramified above $\infty$, a contradiction to the fact that $\pi_1^t(\bb{A}^1)$ is the trivial group. So if $\gamma$ is an even permutation, $G = A_p$.
\end{proof}

\begin{remark}\label{rmk_Abhyankar_deg_p}
In \cite[Section 20]{7} Abhyankar introduced the following cover and calculated its Galois group. Consider the degree-$p$ cover of $\bb{P}^1$ given by the affine equation $\tilde{f} = 0$ where $\tilde{f}(x,y) = y^p - y^t +x$ and consider its Galois closure $\widetilde{Y} \to \bb{P}^1$ with group $G$. Abhyankar showed that for $2 \leq t \leq p-3$, $G = S_p$ if $t$ is even and $G = A_p$ if $t$ is odd. So this is a special case of Proposition \ref{prop_deg_p_cover}.
\end{remark}

Now we construct covers of degree $d \geq p+1$. Similar to the previous case we consider the following assumption.

\begin{ass}\label{ass_num}
Let $p$ be an odd prime, $t \geq 1$ be coprime to $p$ such that $d \coloneqq p+t \geq 5$. Let $r$ and $s$ be two positive integers. Let $n_1$, $\cdots$, $n_s$, $m_1$, $\cdots$, $m_r$ be coprime to $p$ integers such that $\Sigma_{i=1}^s n_i = p+t$, $\Sigma_{l=1}^r m_l = t$. Assume that there exist distinct elements $\alpha_1$, $\cdots$, $\alpha_s$, $\beta_1$, $\cdots$, $\beta_r$ in $k$ such that the polynomial
\begin{equation}\label{eq_g(y)}
g(y) \coloneqq \Pi_{l=1}^r (y-\beta_l) (\Sigma_{i=1}^s n_i \Pi_{j \neq i, 1 \leq j \leq s} (y-\alpha_j)) - \Pi_{i=1}^s (y - \alpha_i) (\Sigma_{l=1}^r m_l \Pi_{u \neq l, 1 \leq u \leq r} (y - \beta_u))
\end{equation}
in $k[y]$ is a non-zero constant in $k$.
\end{ass}

\begin{remark}\label{rmk_num}
In particular if $r=1$, setting $\beta_1 = 0$, Assumption \ref{ass_num} says that there are non-zero distinct elements $\alpha_i$ in $k$, $1 \leq i \leq s$, such that the polynomial

\begin{equation}\label{eq_g(y)_r=1}
y(\Sigma_{i=1}^s n_i \Pi_{j \neq i, 1 \leq j \leq s} (y-\alpha_j)) - t \Pi_{i=1}^s (y - \alpha_i)
\end{equation}
is a non-zero constant in $k$. In terms of coefficients, we need the $\alpha_i$'s to satisfy the following condition for each $1 \leq \nu \leq s-1$.

\begin{equation}\label{eq_g(y)_r=1_coeff}
\Sigma_{1 \leq i_1 < \cdots < i_{s-\nu} \leq s} (n_{i_1} + \cdots + n_{i_{s-\nu}}) \alpha_{i_1} \cdots \alpha_{i_{s-\nu}} = 0.
\end{equation}
When $s=1$ setting $\alpha_1 = 0$ we also get a similar condition on the choice of $\beta_l$'s.
\end{remark}

Before proceeding to the construction of the covers let us see some of the cases where Assumption \ref{ass_num} is satisfied.

\begin{lemma}\label{lem_Ass_holds}
Assumption \ref{ass_num} holds with a choice of distinct $\alpha_i$'s and $\beta_l$'s in the following cases.
\begin{enumerate}
\item $s = 1 = r$ with $(\alpha_1, \beta_1) = (1,0)$;
\item $s=2$, $r=1$ with $(\alpha_1,\alpha_2,\beta_1) = (1, - n_1/n_2, 0)$;
\item $s=1$, $r=2$ with $(\alpha_1,\beta_1, \beta_2) = (0, 1,  - m_1/m_2)$;
\item $s=3$, $r=1$ with $(\alpha_1,\alpha_2, \alpha_3,\beta_1) = (\frac{t+2}{4}, - \frac{t-2}{4}, 1, 0)$ where $(p,t+2) = 1 = (p,t-2)$ and $n_1 = p-2$, $n_2 = 2$, $n_3 = t$;
\item $r = s = 2$ with $(\alpha_1,\alpha_2,\beta_1, \beta_2) = (1, \frac{n_2 - n_1}{2n_2}, 0, \frac{t}{2n_2})$ where $n_i \equiv m_i \tx{ mod }p$ and $n_1 \neq n_2$ in $k$.
\end{enumerate}
\end{lemma}

\begin{proof}
It is easy to see that in each of the cases the assigned values of $\alpha_i$'s and $\beta_l$'s are all distinct and they satisfy Assumption \ref{ass_num}.
\end{proof}

The following result produces our main example of an $S_d$-Galois or an $A_d$-Galois two point branched cover of $\bb{P}^1$.

\begin{proposition}\label{prop_ded_d_cover}
Let $p$ be an odd prime such that $d \coloneqq p+t \geq 5$. Let $r$ and $s$ be two positive integers. Let $n_1$, $\cdots$, $n_s$, $m_1$, $\cdots$, $m_r$ be coprime to $p$ such that $\sum_{i=1}^s n_i = p+t$, $\sum_{l=1}^r m_l = t$. Let  $\alpha_1$, $\cdots$, $\alpha_s$, $\beta_1$, $\cdots$, $\beta_r$ are distinct elements in $k$. Let $\psi \colon Y \to \bb{P}^1$ be the degree-$d$ cover given by the affine equation $f(x,y)=0$ where
\begin{equation}\label{eq_explicit_general_cover}
f(x,y) = \prod_{i=1}^s (y - \alpha_i)^{n_i} - x \prod_{l=1}^r (y - \beta_l)^{m_l}.
\end{equation}
Let $\phi \colon Z \rightarrow \mathbb{P}^1$ be its Galois closure with group $G$.
\begin{enumerate}
\item Then $G$ is a transitive subgroup of $S_d$;\label{itm:1}
\item the cover $\phi$ is tamely ramified with cyclic inertia group generated by an element $\gamma \in G$ of order $\tx{l.c.m.} \{n_1, \cdots, n_s\}$ over $0$, whose disjoint cycle decomposition in $S_d$ consists of $s$ disjoint cycles of length $n_1$, $\cdots$, $n_s$;\label{itm:2}.
\item Over $\infty$, the inertia group is of the form $I = \langle (1, \cdots, p) \rangle \rtimes \langle \theta^i \omega \rangle$ for some $1 \leq i \leq p-1$, where $\theta$ is as in Equation (\ref{eq_theta}) and $\omega \in \tx{Sym}(\{p+1,\cdots,d\})$ is a product of $r$ disjoint cycles of length $m_1$, $\cdots$, $m_r$.
\end{enumerate}
Additionally, if $(\alpha_1,\cdots,\alpha_s,\beta_1,\cdots,\beta_r)$ satisfies Assumption \ref{ass_num}, the cover $\phi$ is \'{e}tale away from $\{0,\infty\}$. Also $\tx{ord}(\theta^i) = \frac{p-1}{(p-1,r+s-1)}$, $|I|=p \times \tx{l.c.m}\{\tx{ord}(\theta^i), \tx{ord}(\omega)\}$, and the upper jump for any $I$-Galois local extension over $\infty$ is given by $\frac{r+s-1}{p-1}$. Moreover, $G$ is a primitive subgroup of $S_d$. Furthermore, if either there is a positive integer $j$ such that $\gamma^j$ is a non-trivial cycle fixing $\geq 3$ points in $\{1,\cdots,d\}$ or if $3 \leq t \leq p-1$,
\[
    G = 
\begin{cases}
    A_d,& \tx{if } \gamma \tx{ is an even permutation}\\
    S_d,              & \tx{if } \gamma \tx{ is an odd permutation.}
\end{cases}
\]
\end{proposition}

\begin{proof}
Since $\alpha_i$'s are distinct from $\beta_l$'s by our assumption and the polynomial $f(x,y)$ is linear in $x$,  it is irreducible in $k(x)[y]$. So $G$ is a transitive subgroup of $S_d$, proving (\ref{itm:1}).

From the equation $f(x,y)=0$ we have $v_{(y-\alpha_i)}(x)=n_i$ for $1 \leq i \leq s$, $v_{(y-\beta_l)}(x^{-1}) = m_l$ for $1 \leq l \leq r$ and $v_{(y^{-1})}(x^{-1}) = p$. Since $\sum_i n_i = p+t$ the fibre $\psi^{-1}(0)$ consists of $s$ points in $Y$ with the ramification index at the point $(y=\alpha_i)$ given by $n_i$, and also since $\sum_l m_l =t$, there are exactly $r+1$ points in $Y$ lying above $\infty$ with ramification indices given by $p$, $m_1$, $\cdots$, $m_r$. Then the description of the inertia groups above $0$ and $\infty$ follows from Proposition \ref{prop_ram_grp}.

Now suppose that Assumption \ref{ass_num} holds. The $y$-derivative of the polynomial $f(x,y)$ is given by

$$f_y(x,y) = \prod_{i=1}^s (y - \alpha_i)^{n_i-1} (\sum_i n_i \prod_{j \neq i} (y - \alpha_j)) - x \prod_{l=1}^r (y - \beta_l)^{m_l-1} (\sum_l m_l \prod_{u \neq l} (y - \beta_u)).$$
Let $(a,b)$ be a common zero of $f$ and $f_y$. Then $0 = f_y(a,b) \prod_{l=1}^r (b - \beta_l) = g(b) \prod_i(b - \alpha_i)^{n_i -1}$ (where $g$ is the polynomial given by Equation \ref{eq_g(y)}). Then $a = 0$ if $n_i \geq 2$ for some $i$ and there is no such common zero otherwise. So the cover $\psi$, and hence $\phi$ is \'{e}tale away from $\{0,\infty\}$. Again by Proposition \ref{prop_ram_grp}, the upper jump is $\frac{r+s-1}{p-1}$ and $\tx{ord}(\theta^i) = \frac{p-1}{(p-1,r+s-1)}$.

Since $G$ is a transitive subgroup of $S_d$ containing the $p$-cycle $\tau$ which fixes $t$ points in $\{1,\cdots,d\}$ and $t < \frac{p+t}{2}$, by \cite[Remark 1.6]{Jones}, $G$ is primitive.

Finally if some power of $\gamma$ is a non-trivial cycle fixing $\geq 3$ points or if $3 \leq t \leq p-1$, by \cite[Theorem 1.2]{Jones}, $G$ contains $A_d$. The rest follows as in Proposition \ref{prop_deg_p_cover}.
\end{proof}

From the above proposition we deduce the following results which will be used later. 

\begin{corollary}\label{cor_t_odd_21case}
Let $p$ be an odd prime, $3 \leq t \leq p-2$ be an odd integer and $d = p+t$. Then there is a connected $A_d$-Galois \'{e}tale cover of the affine line such that $I \coloneqq \langle (1, \cdots, p) \rangle \rtimes \langle \theta^2 (p+1, \cdots, d) \rangle$ occurs as an inertia group at a point above $\infty$.
\end{corollary}

\begin{proof}
Take $s=2$, $r=1$, $n_1 = p+t-1$, $n_2=1$ in Proposition \ref{prop_ded_d_cover}. By Lemma \ref{lem_Ass_holds}(2), Assumption \ref{ass_num} is satisfied. Since $t$ is an odd integer, a $(p+t-1)$-cycle is an even permutation. By Proposition \ref{prop_ded_d_cover}, there is a connected $A_d$-Galois cover of $\bb{P}^1$ branched only at $0$ and $\infty$, over $0$ the inertia groups are generated by conjugates of a $(p+t-1)$-cycle in $A_d$ and $I$ occurs as an inertia group above $\infty$. By Abhyankar's Lemma (\cite[page 279, Expose X, Lemma 3.6]{10}), we obtain a connected $A_d$-Galois cover $\phi \colon Y \to \bb{P}^1$ \'{e}tale away from $\infty$. Since $t$ is odd, we have $(p+t-1,p-1) = 1 = (p+t-1, t)$. So $I$ occurs as an inertia group in the cover $\phi$ above $\infty$.
\end{proof}

\begin{corollary}\label{cor_t_even_12case}
Let $p$ be an odd prime, $4 \leq t \leq p-1$ an integer such that $(t+1,p-1) = 1 = (t-1,p+1)$ and $d = p+t$. Then there is a connected $A_d$-Galois \'{e}tale cover of the affine line such that $I \coloneqq \langle (1, \cdots, p) \rangle \rtimes \langle \theta^2 (p+1, \cdots, d-1) \rangle$ occurs as an inertia group at a point above $\infty$.
\end{corollary}

\begin{proof}
Taking $s=1$, $r=2$, $m_1=t-1$, $m_2=1$ in Proposition \ref{prop_ded_d_cover} and arguing as in Corollary \ref{cor_t_odd_21case} we obtain the result.
\end{proof}

\begin{corollary}\label{cor_t_odd_22case}
Let $p \equiv 2 \pmod{3}$ be an odd prime. Let $3 \leq t \leq p-1$ be an integer, $d = p+t$. Then there is a connected $A_d$-Galois \'{e}tale cover of the affine line such that $I \coloneqq \langle (1, \cdots, p) \rangle \rtimes \langle \theta^{(t,p-1)} (p+1, \cdots, d-1) \rangle$ occurs as an inertia group at a point above $\infty$.
\end{corollary}

\begin{proof}
Take $s=2$, $r=2$, $n_1=p+t-1$, $n_2=1$, $m_1=t-1$, $m_2=1$ in Proposition \ref{prop_ded_d_cover}. By Lemma \ref{lem_Ass_holds}(5), Assumption \ref{ass_num} is satisfied. A $(p+t-1)$-cycle is an odd permutation if and only if $t$ is an even integer. By Proposition \ref{prop_ded_d_cover}, there is a connected $G$-Galois cover of $\bb{P}^1$ branched only at $0$ and $\infty$, where $G = A_d$ if $t$ is odd and $G=S_d$ if $t$ is even. Over $0$ the inertia groups are generated by conjugates of a $(p+t-1)$-cycle in $G$ and $I$ occurs as an inertia group above $\infty$. After a pullback via the $[p+t-1]$-Kummer cover (cf. Definition \ref{def_Kummer_pullback}) we obtain a connected $A_d$-Galois cover $\phi \colon Y \to \bb{P}^1$ \'{e}tale away from $\infty$ such that $I$ occurs as an inertia group in the cover $\phi$ above $\infty$.
\end{proof}

\section{Constructing Covers by Formal Patching}\label{sec_fp}

For the details of the Formal Patching see \cite{2}. The results of this section will be used later. Throughout this section we fix the following notation.

\begin{notation}\label{not+}
Set $R=k[[t]]$, $K=k((t))$. Let $s$ denotes the closed point of $S \coloneqq \tx{Spec}(R)$. For any integral $k$-algebra $A$ with $L = QF(A)$ and any $k$-scheme $W$, set $W_A \coloneqq W \times_k A$, $W_L \coloneqq W_A \times_A L$. For any closed point $w \in W$, $w_A \coloneqq w \times_k A$, $w_L \coloneqq w_A \times_A L$. For any integral scheme $X$ and a closed point $x \in X$ let $K_{X,x} \coloneqq \tx{QF}(\widehat{\cal{O}}_{X,x})$.
\end{notation}

The following result shows that we can construct a family of Galois covers of $S$ starting with Galois covers with smaller Galois groups, provided all that tame parts remain the same.

\begin{lemma}\label{lem_Raynaud_gen}
Let $I$ be an extension of a $p$-group by a cyclic group of order $m$, $p \nmid m$. Let $I_1$ and $I_2$ be two subgroups of $I$ such that $|I_1|/|p(I_1)| = m = |I_2|/|p(I_2)|$. For $i=1$, $2$, assume that $V_i \to S$ be a connected $I_i$-Galois cover which is totally ramified over $s$. Then there is a connected $I$-Galois cover $V \to S \times \bb{A}^1_u$ of integral schemes with branch locus $s \times \bb{A}^1_u$ over which it is totally ramified such that $V \times_{S \times \bb{A}^1_u} S \times \{u=1\} \cong \tx{Ind}_{I_1}^I V_1$ and $V \times_{S \times \bb{A}^1_u} S \times \{u=-1\} \cong \tx{Ind}_{I_2}^I V_2$ as $I$-Galois covers over $S$.
\end{lemma}

\begin{proof}
Let $i \in \{1,2\}$. Then $V_i = \tx{Spec}(A_i)$ for a complete discrete valuation ring $A_i$ with residue field $k$. Let $L_i$ be the field of fractions of $A_i$. After a change of variable we may assume that $L_i^{p(I)} = K' = k((T))$ where the extension $K'/K$ is given by $T^m = t$. Consider the trivial deformation $\tx{Spec}(K'[u]) \to \tx{Spec}(K[u])$ of integral $K$-curves of the cover $\tx{Spec}(K') \to \tx{Spec}(K)$. The compositions
$$\tx{Ind}_{I_1}^I\tx{Spec}(L_1) \to \tx{Spec}(K') \to \tx{Spec}(K) = \tx{Spec}(K[u]) \times_k (u=1),$$
$$\tx{Ind}_{I_2}^I\tx{Spec}(L_2) \to \tx{Spec}(K') \to \tx{Spec}(K) = \tx{Spec}(K[u]) \times_k (u=-1)$$
are (possibly disconnected) $I$-Galois covers. By applying \cite[Proposition 3.11]{Ha_extn} with $\Gamma = I$, $G =\bb{Z}/m$, $X$ as the affine $u$-line over $K$ and $X'$ as the closed subset of $X$ consisting of the points $(u=1)$ and $(u=-1)$, we obtain a connected $p(I)$-Galois cover $V' \to \bb{A}^1_{K'}$ of integral curves such that the composition $f \colon V' \to \bb{A}^1_{K}$ is a connected $I$-Galois cover and such that $V' \times_{\bb{A}^1_u} (u=1) \cong \tx{Ind}_{I_1}^I\tx{Spec}(L_1)$ and $V' \times_{\bb{A}^1_u} (u=-1) \cong \tx{Ind}_{I_2}^I\tx{Spec}(L_2)$ as $I$-Galois covers of $\tx{Spec}(K)$. Take $V$ to be the normalization of $S \times \bb{A}^1_u$ in the function field of $V'$. Then the cover $V \to S \times \bb{A}^1_u$ satisfies the stated properties.
\end{proof}

We obtain the following result using the Formal patching results from \cite{Ha_AC} and \cite{DK} and the above lemma provides a compatibility of local extensions. Similar results appeared in the literature for the purely wild ramification (cf. \cite[Theorem 2.2.3]{3}).

\begin{theorem}\label{thm_Raynaud_gen}
Let $G$ be a finite group, $I \subset G$ be an extension of a $p$-group by a cyclic group of order $m$, $p \nmid m$. Let $G_1$, $G_2$ be subgroups of $G$. Let $X$ be a smooth projective connected $k$-curve. Assume the existence of the following Galois covers.
\begin{enumerate}
\item Let $f_1 \colon Y_1 \to X$ be a connected $G_1$-Galois cover \'{e}tale away from $B_1 \subset X$. Let $x \in B_1$ be a closed point and let $I_1$ occurs as an inertia group above $x$. For $x_1 \neq x$ in $B_1$ let $J_{x_1}$ occurs as an inertia group above $x_1$.
\item Let $f_2 \colon Y_2 \to \bb{P}^1$ be a connected $G_2$-Galois cover with branch locus $B_2 \subset \bb{P}^1$, $0 \in B_2$. Let $I_2$ occurs as an inertia group above $0$ and for points $y \neq 0$ in $B_2$ let $I'_y$ occurs as an inertia group above $y$.
\end{enumerate}
Suppose that $I_1$ and $I_2$ are subgroups of $I$ such that $|I_1/p(I_1)| = m = |I_2/p(I_2)|$. If $G = \langle G_1, G_2, I \rangle$, then there is subset $B_0 \subset X$ disjoint from $B_1$ with $|B_0| = |B_2|-1$ and a connected $G$-Galois cover of $X$ \'{e}tale away from $B_1 \sqcup B_0$ such that $I$ occurs as an inertia group above $x$, for $x_1 \neq x$ in $B_1$, $J_{x_1}$ occurs as an inertia group above $x_1$ and for $x_0$ in $B_0$, $I'_{x_0}$ occurs as an inertia group above $x_0$.
\end{theorem}

\begin{proof}
Let $y_1 \in Y_1$ with $f_1(y_1) = x$ such that $I_1$ is the Galois group of the field extension $K_{Y_1,y_1}/K_{X,x}$. Also let $y_2 \in Y_2$ be a point lying above $\infty$ such that $I_2$ is the Galois group of the field extension $K_{Y_2,y_2}/k((x))$. We identify $\widehat{\mathcal{O}}_{X,x}$ with $k[[x]]$. Then taking $V_1 = \tx{Spec}(\widehat{\mathcal{O}}_{Y_1,y_1}) \to \tx{Spec}(k[[x]])$ and $V_2 = \tx{Spec}(\widehat{\mathcal{O}}_{Y_2,y_2}) \to \tx{Spec}(k[[x]])$ in Lemma \ref{lem_Raynaud_gen} we obtain a connected $I$-Galois cover $g \colon V \to \tx{Spec}(k[[x]][u])$ of integral schemes with branch locus $x \times \bb{A}^1_u$ over which it is totally ramified and such that $V \times_{\tx{Spec}(k[[x]][u])} \tx{Spec}(k[[x]]) \times \{u=1\} \cong \tx{Ind}_{I_1}^I V_1$ and $V \times_{\tx{Spec}(k[[x]][u]} \tx{Spec}(k[[x]]) \times \{u=-1\} \cong \tx{Ind}_{I_2}^I V_2$ as $I$-Galois covers over $\tx{Spec}(k[[x]])$. Then the cover $f_1$ and the $I$-Galois cover $g$ satisfy the hypothesis of \cite[Lemma 3.2]{DK}. By \cite[Lemma 3.3]{DK}, there is an open dense set $\cal{W}_1 \subset \bb{A}^1_u$ such that for all closed point $(u=\beta)$ in $\cal{W}_1$ the following holds. There is a connected $H_1 \coloneqq \langle G_1, I \rangle$-Galois cover $h_1 \colon Z_1 \to X$ \'{e}tale away from $B_1$ such that there is a point $z_1 \in Z_1$ above $x$ for which $\tx{Spec}(\widehat{\cal{O}}_{Z_1,z_1})$ is equal to the fibre of the cover $g \colon V \to \tx{Spec}(k[[x]][u])$ over $(u=\beta)$ as $I$-Galois covers over $\tx{Spec}(k[[x]])$, and for any point $x_1 \neq x$ in $B_1$, $J_{x_1}$ occurs as an inertia group above $x_1$. Similarly, the covers $f_2$ and $g$ satisfy the hypothesis of \cite[Lemma 3.2]{DK} and by \cite[Lemma 3.3]{DK}, there is an open dense set $\cal{W}_2 \subset \bb{A}^1_u$ such that for all closed point $(u=\alpha)$ in $\cal{W}_2$ the following holds. There is a connected $H_2 \coloneqq \langle G_2, I \rangle$-Galois cover $h_2 \colon Z_2 \to \bb{P}^1$ \'{e}tale away from $B_2$ such that there is a point $z_2 \in Z_2$ above $0$ for which $\tx{Spec}(\widehat{\cal{O}}_{Z_2,z_2})$ as an $I$-Galois cover of $\tx{Spec}(k[[x]])$ is equal to the fibre of $g$ over $(u=\alpha)$. We fix a closed point $(u=a)$ in $\cal{W}_1 \cap \cal{W}_2$ and consider the corresponding covers $h_1 \colon Z_1 \to X$ and $h_2 \colon Z_2 \to \bb{P}^1$ as above. Then $\tx{Spec}(\widehat{\cal{O}}_{Z_1,z_1})$ and $\tx{Spec}(\widehat{\cal{O}}_{Z_2,z_2})$ are isomorphic $I$-Galois covers of $\tx{Spec}(k[[x]])$.

Let $T^*$ be a regular irreducible projective $R$-curve with generic fibre $X_K$ together with a cover $T^* \to \bb{P}^1_R$ and whose closed fibre $T'$ is the union of two irreducible components $X$ and $\bb{P}_y^1$ meeting at a point $\eta$ and such that the complete local ring of $T^*$ at $\eta$ is given by $\widehat{\cal{O}}_{T^*,\tau} = k[[x,y]][t]/(t-xy) \cong k[[x,y]]$. Now consider the trivial deformation $\tx{Spec}(\widehat{\cal{O}}_{Z_2,z_2}[y]) \to \tx{Spec}(k[[x]][y])$. Let $\widehat{N}^* \coloneqq \tx{Spec}(\widehat{\cal{O}}_{Z_2,z_2}[y]) \times_{\tx{Spec}(k[[x]][y])} \tx{Spec}(k[[x,y]])$. Let $X-x = \tx{Spec}(A)$, $Z_1 - z_1 = \tx{Spec}(B_1)$ and $Z_2 - z_2 = \tx{Spec}(B_2)$. Then the hypothesis of \cite[Proposition 2.3]{Ha_AC} is satisfied with the covers $W_1^{'*} = \tx{Spec}(B_1[[t]]) \to X^{'*}_1 = \tx{Spec}A[[t]]$ and $W_2^{'*} = \tx{Spec}(B_2[[t]]) \to X_2^{'*} = \tx{Spec}(k[y^{-1}][[t]])$ induced by $h_1$ and $h_2$, respectively, and with the isomorphisms $\widehat{N}^* \times_{\tx{Spec}(k[[x,y]])} \tx{Spec}(K_{X,x}[[t]]) \cong \tx{Spec}(K_{Z_1,z_1})$ and $\widehat{N}^* \times_{\tx{Spec}(k[[x,y]])} \tx{Spec}(k((y))[[t]]) = \tx{Spec}(K_{Z_2,z_2})$. By loc. cit. there is an irreducible normal $G$-Galois cover $h^* \colon V^* \to T^*$ such that $V^* \times_{T^*} X^{'*}_1 \cong \tx{Ind}_{H_1}^G W^{'*}_1$, $V^* \times_{T^*} X^{'*}_2 \cong \tx{Ind}_{H_2}^G W^{'*}_2$ and $V^* \times_{T^*} \widehat{\cal{O}}_{T^*,\eta} \cong \tx{Ind}_I^G\widehat{N}^*$ as $G$-Galois covers. Consider the generic fibre $h^0 \colon V^0 \to X_K$ of the cover $h^*$. Then there is a set $B_0 \in X$ disjoint from $B_1$ with $|B_0| = |B_2 - 0|$ such that $h^0$ is \'{e}tale away from $\{x'_K | x' \in B_1 \sqcup B_0\}$, $I$ occurs as an inertia group above $x_K$, for $x_1 \neq x$ in $B_1$, $J_{x_1}$ occurs as an inertia group above $x_{1,K}$ and for $x_0$ in $B_0$, $I'_{x_0}$ occurs as an inertia group above $x_{0,K}$. Since $Z_1 \times_R K$, $Z_2 \times_R K$ and $T \times_R K$ are smooth over $K$, $V_0$ is also smooth over $K$. Also since $T'$ is generically smooth and the cover $V^* \to T^*$ is generically unramified, the closed fibre $V^* \times_{T^*} T' \to T'$ of $h^*$ is generically smooth. Now the result follows by \cite[Corollary 2.7]{Ha_AC}.
\end{proof}

By induction on $n$ and using the above theorem, we obtain the following result which generalizes a patching result by Raynaud (\cite[Theorem 2.2.3]{3}). This result allows us to construct Galois covers with control on the inertia groups from covers with smaller Galois groups. As an application (Lemma \ref{lem_Alt_FP}) of it we will see that we can restrict ourselves to a fewer cases of proving the IC.

\begin{corollary}\label{cor_Raynaud_gen}
Let $n \geq 2$ be an integer. Let $G$ be a finite group, $I \subset G$ be an extension of a $p$-group by a cyclic group of order $m$, $p \nmid m$. For $1 \leq i \leq n$, let $G_i$ be subgroups of $G$, $I_i$ be a subgroup of $I$ with $|I_i/p(I_i)| = m$ such that the following hold.
\begin{enumerate}
\item $X$ is a smooth projective connected $k$-curve, $f \colon Y \to X$ is a connected $G_1$-Galois cover \'{e}tale away from $B \subset X$. Let $x \in B$ be a closed point and let $I_1$ occurs as an inertia group above $x$. For $x' \neq x$ in $B$ let $J_{x'}$ occurs as an inertia group above $x'$.
\item For each $2 \leq i \leq n$ the pair $(G_i,I_i)$ is realizable.
\end{enumerate}
If $G = \langle G_1, \cdots, G_n, I \rangle$ then there is a connected $G$-Galois cover $Z \to X$ \'{e}tale away from $B$ such that $I$ occurs as an inertia group above $x$ and for $x' \neq x$ in $B$, $J_{x'}$ occurs as an inertia group above $x'$.
\end{corollary}

For our next result let us fix the following notation.

\begin{notation}\label{not_for_pdt}
Let $G_1$, $G_2$ be two finite groups, $X$ a smooth projective connected $k$-curve. Let $B \subset X$ be a finite set of closed points. Assume that for $i=1$, $2$, there is a connected $G_i$-Galois cover $f_i \colon Y_i \to X$ \'{e}tale away from $B$ such that a $p$-group (possibly trivial) $P_{x,i}$ occurs as the inertia group above $x \in B$. For each $x \in B$, let $Q_x$ be a $p$-group (possibly trivial), and for $i=1$, $2$, let $N_{x,i}$ be a normal subgroup of $P_{x,i}$ such that $P_{x,i}/N_{x,i} \cong Q_x$.
\end{notation}

The following result shows that in the set up of the above Notation \ref{not_for_pdt}, certain kind of field extensions can be realized as local extensions by the Galois covers for both the groups $G_1$ and $G_2$. This will be used (Theorem \ref{thm_pdt_arbit_groups}) to show that the GPWIC is true for certain product of groups. As in Theorem \ref{thm_Raynaud_gen}, we again use a technique from \cite{DK} to realize a common local extension at a point over $X$ for two different covers.

\begin{lemma}\label{lem_common_cover_for_product}
Assume that Notation \ref{not_for_pdt} hold. Then for $i=1$, $2$, there is a connected $G_i$-Galois cover $Z_i \to X$ \'{e}tale away from $B$ such that $P_{x,i}$ occurs as an inertia group above $x \in B$ and such that there is a point $z_i \in Z_i$ over $x$ with $K_{Z_1,z_1}^{N_{x,1}}/K_{X,x} \cong K_{Z_2,z_2}^{N_{x,2}}/K_{X,x}$ as $Q_x$-Galois extensions of $K_{X,x}$.
\end{lemma}

\begin{proof}
Let $x \in B$, $i = 1$, $2$. Without loss of generality we may assume that all the groups $P_{x,i}$ and $Q_x$ are non-trivial. Let $y_i \in Y_i$ with $f_i(y_i)=x$ such that $P_{x,i}$ is the Galois group of the field extension $K_{Y_i,y_i}/K_{X,x}$. Take $I_1 = I_2 = I = Q_x$ and $V_i = \tx{Spec}(\widehat{\cal{O}}_{Y_i,y_i}^{N_{x,i}})$ in Lemma \ref{lem_Raynaud_gen}. Then we obtain a connected $Q_x$-Galois cover $V \to \tx{Spec}(\widehat{\cal{O}}_{X,x}) \times \bb{A}^1_u$ of integral schemes with branch locus $x \times \bb{A}^1_u$ over which it is totally ramified such that $V \times_{\bb{A}^1_u} \{u=1\} \cong V_1$ and $V \times_{\bb{A}^1_u} \{u=-1\} \cong V_2$ as $Q_x$-Galois covers over $\tx{Spec}(\widehat{\cal{O}}_{X,x})$. By \cite[Theorem 3.11]{Ha_extn}, there is a connected $P_{x,i}$-Galois cover $g_i \colon W_i \to \tx{Spec}(\widehat{\cal{O}}_{X,x}) \times \bb{A}^1_u$ of integral schemes dominating $V \to \tx{Spec}(\widehat{\cal{O}}_{X,x}) \times \bb{A}^1_u$ such that the fibre of $g_1$ over $(u=1)$ is $\tx{Spec}(\widehat{\cal{O}}_{Y_1,y_1})$ and the fibre of $g_2$ over $(u=-1)$ is $\tx{Spec}(\widehat{\cal{O}}_{Y_2,y_2})$. Then by \cite[Lemma 3.2]{DK}, there are connected $G_i$-Galois covers of $X_R$ satisfying the hypothesis of \cite[Lemma 3.3]{DK}. Using \cite[Remark 3.4]{DK}, for each $x \in B$, $i \in \{1,2\}$, choose a dense open set $\cal{W}_{x,i}\subset \bb{A}^1_u$ and fix a closed point in $\cal{W}_{x,1} \cap \cal{W}_{x,2}$. Then for $i=1$, $2$, there is a connected $G_i$-Galois cover $Z_i \to X$ \'{e}tale away from $B$ such that $P_{x,i}$ occurs as an inertia group above $x \in B$ and such that there is a point $z_i \in Z_i$ over $x$ with $K_{Z_1,z_1}^{N_{x,1}}/K_{X,x} \cong K_{Z_2,z_2}^{N_{x,2}}/K_{X,x}$ as $Q_x$-Galois extensions.
\end{proof}

We recall and restate the following theorem due to Harbater which will be used throughout Section \cref{sec_GPWIC}--\cref{sec_que_evidence}.

\begin{theorem}\label{thm_patching_covers}
(\cite[Corollary to Patching Theorem]{Ha_two_patch}) Let $r \geq 1$ be an integer. Let $G$ be a finite group, $G_1$ and $G_2$ be two subgroups of $G$ such that $G = \langle G_1 , G_2 \rangle$. Let $X$ be a smooth projective connected $k$-curve. Let $B$ be a finite set of closed points of $X$ containing a point $x_0$. Let $B' \coloneqq \{\eta_0, \cdots, \eta_r\}$ be a set of distinct points of $\bb{P}^1$. Let $a \in G_1 \cap G_2$ be an element of order prime-to-$p$. Assume that
\begin{enumerate}
\item there is a connected $G_1$-Galois cover $f \colon Y \to X$ \'{e}tale away from $B$ such that $I_x$ occurs as an inertia group above $x \in B$ and $I_{x_0} = \langle a \rangle$;
\item there is a connected $G_2$-Galois cover $g \colon W \to \bb{P}^1$ \'{e}tale away from $B'$ such that $J_i$ occurs as an inertia group above $\eta_i$, $1 \leq i \leq r$, and such that $J_0 = \langle a^{-1} \rangle$.
\end{enumerate}
Then there is a set $B''= \{x_1, \cdots, x_r\}$ of closed points of $X$ disjoint from $B$ and a connected $G$-Galois cover of $X$ \'{e}tale away from $B \sqcup B''$ such that $I_x$ occurs as an inertia group above $x \in B \setminus \{x_0\}$, $J_r$ occurs as an inertia group above $x_0$ and $J_i$ occurs as an inertia group above $x_i$, $1 \leq i \leq r-1$.
\end{theorem}

In \cite{AC_embed_Harbater} a special case of the following result appeared which also solves the split quasi-$p$ embedding problem.

\begin{theorem}\label{thm_embedding}
Let $G$ be a finite group, $\psi \colon Y \to X$ be a smooth connected $G$-Galois cover. Let $x_0 \in X$ be a closed point. Let $\Gamma$ be a finite group generated by $G$ and a quasi $p$-group $H$ such that there is a $p$-subgroup $P$ of $H$ which is normalized by $G$ and such that the pair $(H,P)$ is realizable. Then there is a $\Gamma$-Galois cover $\phi \colon Z \to X$ of smooth projective connected $k$-curves dominating the cover $\psi$ such that the following hold.
\begin{enumerate}
\item For a closed point $x \neq x_0$ if $I_x$ occurs as an inertia group at a point over $x$ for the cover $\psi$, then $I_x$ also occurs as an inertia group at a point over $x$ for the cover $\phi$;
\item if $I_0$ occurs as an inertia group at a point above $x_0$ for the cover $\psi$, $I_0 P$ occurs as an inertia group at a point above $x_0$ for the cover $\phi$;
\item the covers $Z/ \langle H^\Gamma \rangle \to X$ and $Y/(G \cap \langle H^\Gamma \rangle) \to X$ are isomorphic as $\Gamma/ \langle H^\Gamma \rangle$-Galois covers of $X$.
\end{enumerate}
\end{theorem}

\begin{proof}
By \cite[Theorem 2.1, Theorem 4.1]{AC_embed_Harbater} the above conclusion holds when $P$ is replaced by a Sylow $p$-subgroup of $H$ which is normalized by $G$. But the same proof works under our hypothesis with the additional assumption that there is a connected $H$-Galois \'{e}tale cover of the affine line such that $P$ occurs as an inertia group above $\infty$.
\end{proof}

\section{Inertia Conjecture for the Alternating groups}\label{sec_alt_IC}
In this section we prove the Inertia Conjecture (Conjecture \ref{IC}) for some Alternating groups. Recall that the IC was proved to be true for $A_p$ (\cite[Theorem 1.2]{BP}) and when $p \equiv 2 \pmod{3}$ for $A_{p+2}$(\cite[Theorem 1.2]{8}). We show that when $p \equiv 2  \pmod{3}$ the IC is true for the groups $A_{p+1}$, $A_{p+3}$ and $A_{p+4}$, and with some extra condition on $p$ the IC is also true for $A_{p+5}$. The covers will be constructed using the results and techniques from Section \cref{sec_explicit_eq} and Corollary \ref{cor_Raynaud_gen} from the previous section. Throughout this section $\tau$ denotes the $p$-cycle $(1,\cdots,p)$ in $S_p$.

In view of Equation (\ref{eq_theta}), to prove the IC for $A_d$, $p < d < 2p$, we need to prove that there is a connected $A_d$-Galois \'{e}tale cover of the affine line such that $I = \langle \tau \rangle \rtimes \langle \theta^i \omega \rangle$ occurs as an inertia group above $\infty$ for every $1 \leq i \leq p-1$ and $\omega \in \tx{Sym}(\{p+1,\cdots,d\})$ such that $\theta^i \omega$ is an even permutation. Note that for $\theta^i \omega \in A_d$, $i$ an even integer if and only if $\omega$ is an even permutation. Now we make some observations which reduces the proof of the IC to the realization of the pair $(A_d,I)$ to a fewer cases.

\begin{remark}\label{rmk_reduction_Alt}
Since $\langle \theta^i \omega \rangle = \langle \theta^{(i,p-1)} \omega \rangle$, it is enough to consider the $i$'s dividing $p-1$. Also using Abhyankar's Lemma (\cite[XIII, Proposition 5.2]{10}) it is enough to prove for the cases when $I$ is a maximal inertia group in the sense of \cite[Section 4.9]{survey_paper}.
\end{remark}

In fact the following result shows that it is enough to consider the more restricted cases when $\omega$ acts on the set $\{p+1,\cdots,d\}$ either with one fixed point or without any fixed point, provided the IC is true for the Alternating groups of lower degree.

\begin{lemma}\label{lem_Alt_FP}
Let $p$ be an odd prime, $p+1 \leq d \leq 2p-1$. Assume that the pair $(A_d,I)$ is realizable for a subgroup $I \subset A_d$ which fixes $\geq 1$ points in $\{1,\cdots,d\}$. Then for all $d' \geq d$, the pair $(A_{d'}, I)$ is also realizable.

In particular, to prove the IC for $A_d$, $p+2 \leq d \leq 2p-1$, it is enough to prove that the IC is true for each $A_{u}$, $p \leq u \leq d-2$, and that the pair $(A_d,I)$ is realizable for each $I \subset N_{A_{d}}(\langle \tau \rangle)$ such that $I$ fixes $0$ or $1$ point in the set $\{1,\cdots,d\}$.
\end{lemma}

\begin{proof}
This is a direct consequence of Corollary \ref{cor_Raynaud_gen} with $G_i = \tx{Alt}(\tx{Supp}(I) \cup S_i)$ for every set $S_i \subset \{1,\cdots, d'\} \setminus \tx{Supp}(I)$ of size $d - |\tx{Supp}(I)|$. The second statement follows from the first one.
\end{proof}

We are now ready to prove that IC for certain Alternating groups.

\begin{theorem}\label{thm_A_p+1}
Let $p \equiv 2 \pmod{3}$ be an odd prime. Then the IC is true for the group $A_{p+1}$.
\end{theorem}

\begin{proof}
By Remark \ref{rmk_reduction_Alt}, it is enough to prove that the pair $(A_{p+1},I)$ is realizable for $I = \langle \tau \rangle \rtimes \langle \theta^2 \rangle$. Set $s=3$, $r=1$, $n_1 = p-2$, $n_2 = 2$, $n_3 = 1$. By Lemma \ref{lem_Ass_holds}(4), Assumption \ref{ass_num} holds for the choice $(\alpha_1,\alpha_2,\alpha_3,\beta_1) = (3/4,1/4,1,0)$. Let $\psi \colon Y \to \bb{P}^1$ be the degree-$(p+1)$ cover given by the affine equation
\begin{equation}\label{eq_p+1}
\Pi_{i=1}^3 (y - \alpha_i)^{n_i} - xy = 0.
\end{equation}
Let $\phi \colon Z \to \bb{P}^1$ be the Galois closure of the cover $\psi$ with Galois group $G$. Then by Proposition \ref{prop_ded_d_cover}, $G$ is a primitive subgroup of $S_{p+1}$ and the cover $\phi$ is \'{e}tale away from $\{0, \infty\}$ such that $\langle (1,\cdots,p-2)(p-1,p) \rangle$ occurs as an inertia group over $0$ and $\langle \tau \rangle \rtimes \langle \theta \rangle$ occurs as an inertia group over $\infty$. Since $p-2$ is odd, $G$ contains the transposition $(p-1,p)$ and since $G$ also contains the $p$-cycle $\tau$, by \cite[Lemma 4.4.3]{TGT}, $G = S_{p+1}$. After the pullback of $\phi$ under the $[2(p-2)]$-Kummer cover we obtain a connected $A_{p+1}$-Galois \'{e}tale cover of the affine line such that $I$ occurs as an inertia group above $\infty$.
\end{proof}

\begin{theorem}\label{thm_A_p+3}
Let $p \equiv 2 \pmod{3}$ be an odd prime. Then the IC is true for the group $A_{p+3}$.
\end{theorem}

\begin{proof}
When $p \equiv 2 \pmod{3}$, by Abhyankar's Lemma it is enough to prove that there is a connected $A_{p+3}$-Galois cover of $\bb{P}^1$ \'{e}tale away from $\infty$ such that $I = \langle \tau \rangle \rtimes \langle \beta \rangle$ occurs as an inertia group at a point above $\infty$ where $\beta$ is of the form $\beta = \theta^2 (p+1, p+2, p+3)$ or $\beta = \theta (p+1, p+2)$. These are immediate from Corollary \ref{cor_t_odd_21case} and Corollary \ref{cor_t_odd_22case}.
\end{proof}

\begin{theorem}\label{thm_A_p+4}
Let $p \equiv 2 \pmod{3}$ be an odd prime. Then the IC is true for the group $A_{p+4}$.
\end{theorem}

\begin{proof}
Let $p \equiv 2 \pmod{3}$ be an odd prime. By Abhyankar's Lemma, to prove the IC for $A_{p+4}$ it is enough to prove that there is a connected $A_{p+4}$-Galois cover of $\bb{P}^1$ \'{e}tale away from $\infty$ such that $I = \langle \tau \rangle \rtimes \langle \beta \rangle$ occurs as an inertia group at a point above $\infty$ where $\beta$ is of the form $\beta = \theta (p+1, p+2)$ or $\beta = \theta^2 (p+1, p+2, p+3)$ or $\beta = \theta (p+1, p+2, p+3, p+4)$.

Since the IC is true for $A_{p+2}$, by Lemma \ref{lem_Alt_FP}, the pair $(A_{p+4},\langle \tau \rangle \rtimes \langle \theta (p+1,p+2) \rangle)$ is realizable.

Now fix an element $w_3 \in k$ such that $w_3^2 = 3$. Then for $s=2$, $r=2$, $n_1 = p+2$, $n_2 = 2$, $m_1 = 3$, $m_2 = 1$, Assumption \ref{ass_num} holds for the choice $(\alpha_1,\alpha_2,\beta_1,\beta_2) = (\frac{1+w_3}{4}, \frac{1- w_3}{4}, 0, 1)$. So we can apply Proposition \ref{prop_ded_d_cover} to obtain a connected $S_{p+4}$-Galois cover of $\bb{P}^1$ \'{e}tale away from $\{0,\infty\}$ such that $\langle (1,\cdots,p+2)(p+3,p+4)\rangle$ occurs as an inertia group above $0$ and $\langle \tau \rangle \rtimes \langle \theta (p+1,p+2,p+3) \rangle$ occurs as an inertia group above $\infty$. After a $[2(p+2)]$-Kummer pullback we obtain a connected $A_{p+4}$-Galois \'{e}tale cover of the affine line such that $\langle \tau \rangle \rtimes \langle \theta^2 (p+1,p+2,p+3) \rangle$ occurs as an inertia group above $\infty$.

For the last case fix an element $w_2 \in k$ such that $w_2^2 = 2$. Then for $s=3$, $r=1$, $n_1 = p-2$, $n_2 = n_3 =3$,  Assumption \ref{ass_num} holds for the choice $(\alpha_1,\alpha_2,\alpha_3,\beta_1) = (1, \frac{1+w_2}{3}, \frac{1-w_2}{3}, 0)$. By Proposition \ref{prop_ded_d_cover}, there is a connected $A_{p+4}$-Galois cover of $\bb{P}^1$ \'{e}tale away from $\{0,\infty\}$ such that $\langle (1,\cdots,p-2)(p-1,p,p+1)(p+2,p+3,p+4)\rangle$ occurs as an inertia group above $0$ and $\langle \tau \rangle \rtimes \langle \theta (p+1,p+2,p+3,p+4) \rangle$ occurs as an inertia group above $\infty$. So by Abhyankar's Lemma, the pair $(A_{p+4}, \langle \tau \rangle \rtimes \langle \theta (p+1,p+2,p+3,p+4) \rangle)$ is realizable.
\end{proof}

\begin{lemma}\label{lem_A_p+5}
When $p \equiv 2 \pmod{3}$ is a prime $>5$, the pair $(A_{p+5},I_i)$ is realizable for $2 \leq i \leq 5$, where $I_2 = \langle \tau \rangle \rtimes \langle \theta (p+1, p+2) \rangle$, $I_3 = \langle \tau \rangle \rtimes \langle \theta^2 (p+1, p+2, p+3) \rangle$, $I_4 = \langle \tau \rangle \rtimes \langle \theta (p+1, p+2, p+3, p+4) \rangle$, $I_5 = \langle \tau \rangle \rtimes \langle \theta^2 (p+1, p+2, p+3, p+4, p+5) \rangle$. Additionally if $4 \nmid (p+1)$, the pair $(A_{p+5},\langle \tau \rangle \rtimes \langle \theta (p+1, p+2)( p+3, p+4, p+5) \rangle )$ is also realizable.
\end{lemma}

\begin{proof}
Let $p \equiv 2 \pmod{3}$ be a prime $> 5$. Since the IC is true for $A_{p+2}$ (\cite[Theorem 1.2]{8}) and for $A_{p+3}$ (Theorem \ref{thm_A_p+3}), by Lemma \ref{lem_Alt_FP} the first two cases follow.

Now we consider the realization of $I_4$ as an inertia group. Fix $w_{2/3} \in k$ such that $w_{2/3}^2 = 2/3$. Then for $s=2$, $r=2$, $n_1 = p+2$, $n_2 = 3$, $m_1 = 4$, $m_2 = 1$, Assumption \ref{ass_num} holds for the choice $(\alpha_1,\alpha_2,\beta_1,\beta_2) = (\frac{1 - 3 w_{2/3}}{5}, \frac{1 + 2 w_{2/3}}{5}, 0, 1)$. So we can apply Proposition \ref{prop_ded_d_cover} to obtain a connected $A_{p+5}$-Galois cover of $\bb{P}^1$ \'{e}tale away from $\{0,\infty\}$ such that $\langle (1,\cdots,p+2)(p+3,p+4,p+5)\rangle$ occurs as an inertia group above $0$ and $\langle \tau \rangle \rtimes \langle \theta (p+1,p+2,p+3,p+4) \rangle$ occurs as an inertia group above $\infty$. By Abhyankar's Lemma, the pair $(A_{p+5}, I_4)$ is realizable.

For the next case, when $(5,p-1)=1$, the pair $(A_{p+5}, I_5)$ is realizable by Corollary \ref{cor_t_odd_21case}. So let $(5,p-1)=5$. Fix $w_7 \in k$ such that $w_7^2 = 7$. Then for $s=3$, $r=1$, $n_1 = p-2$, $n_2 = 6$, $n_3=1$, $m_1 = 5$, Assumption \ref{ass_num} holds for the choice $(\alpha_1,\alpha_2, \alpha_3, \beta_1) = (\frac{w_7 + 1}{2}, \frac{w_7 - 1}{6}, 2, 0)$. By Proposition \ref{prop_ded_d_cover}, there is a connected $S_{p+5}$-Galois cover of $\bb{P}^1$ \'{e}tale away from $\{0,\infty\}$ such that $\langle (1,\cdots,p-2)(p-1,p,p+1,p+2,p+3,p+4)\rangle$ occurs as an inertia group above $0$ and $\langle \tau \rangle \rtimes \langle \theta (p+1,p+2,p+3,p+4,p+5) \rangle$ occurs as an inertia group above $\infty$. Since $(p-1,5)=5$, we have $(p-2,5) =1$. So after a $[6(p-2)]$-Kummer pullback we obtain a connected $A_{p+5}$-Galois \'{e}tale cover of the affine line such that $I_5$ occurs as an inertia group above $\infty$.

Now we consider the last case with the additional assumption $(p+1,4) = 2$. Set $s = 2 = r$, $m_1 = 2$, $m_2 = 3$, $n_1 = n_2 = \frac{p+5}{2}$. Choose an element $w_3 \in k$ such that $w^2=3$. Then for $(\alpha_1,\alpha_2,\beta_1,\beta_2) = (\frac{3+2w_3}{5},\frac{3-2w_3}{5},0,1)$, Assumption \ref{ass_num} holds. So we can apply Proposition \ref{prop_ded_d_cover} to obtain a connected $A_{p+5}$-Galois cover of $\bb{P}^1$ branched only at $0$ and $\infty$ such that $\langle (1,\cdots,\frac{p+5}{2}) (\frac{p+5}{2}+1,\cdots,p+5)\rangle$ occurs as an inertia groups above $0$ and $I = \langle \tau \rangle \rtimes \langle \theta (p+1, p+2) (p+3, p+4, p+5) \rangle)$ occurs as an inertia group at a point over $\infty$. Then using Abhyankar's Lemma we see that the pair $(A_{p+5}, I)$ is realizable.
\end{proof}

Using the above lemma and Abhyankar's Lemma (\cite[XIII, Proposition 5.2]{10}) we conclude the following result.

\begin{theorem}\label{thm_A_p+5}
Let $p \equiv 2 \pmod{3}$ be a prime $\geq 17$ such that $4 \nmid (p+1)$. Then the IC is true for the group $A_{p+5}$.
\end{theorem}

\section{Generalizations of the Inertia Conjecture}\label{sec_questions}
In this section we pose certain questions which generalize the Inertia Conjecture. Although these generalizations make sense for $p=2$, we restrict ourselves to $p > 2$. Consider the following notation for the rest of this section.

\begin{notation}\label{not_general_set_up}
Let $r \geq 1$, $X$ be a smooth projective connected $k$-curve, $G$ be a finite group. Let $B = \{x_1, \cdots, x_r\} \subset X$ be a set of closed points in $X$. Let $P_1$, $\cdots$, $P_r$ be $p$-subgroups of $G$, possibly trivial. Define a subnormal series $\{H_j\}_{j \geq 1}$ of $G$ inductively as follows. Set $H_0 \coloneqq G$, $H_{j+1} \coloneqq \langle P_i^{H_{j}} | 1 \leq i \leq r \rangle \subset G$. Then each $H_j$ is a normal subgroup (normal quasi $p$-subgroup if $P_i$ is non-trivial for some $i$) of $H_{j+1}$ containing all the $P_i$'s, and since $G$ is a finite group there is a minimal non-negative integer $l$ such that $H_j = H_l$ for all $j \geq l$.
\end{notation}

Now under the above notation, suppose that $\phi \colon Z \to X$ is a connected $G$-Galois cover \'{e}tale away from $B$ and let $I_i$ occurs as an inertia group above $x_i$, $1 \leq i \leq r$. Since $H_1$ is a normal subgroup of $G$, the cover $\phi$ factors through a connected $G/H_1$-Galois cover $\psi_1 \colon Y_1 \to X$ \'{e}tale away from $B$ such that $I_i/I_i \cap H_1$ occurs as an inertia group above $x_i$, $1 \leq i \leq r$. Inductively for $0 \leq j \leq l$, the $H_j$-Galois cover $Z \to Y_j$ factors through a connected $H_j/H_{j+1}$-Galois cover $\psi_{j+1} \colon Y_{j+1} \to Y_j$ \'{e}tale away from $B$. Moreover, if $y_{i,j}$ is a point of $Y_j$ lying above $x_i$, then $I_i \cap H_{j-1}/I_i \cap H_j$ occurs as an inertia group above $y_{i,j-1}$ in the cover $\psi_j$, $1 \leq i \leq r$. So $\phi$ is the composition of a tower
\[
\xymatrix{
Z \ar[r] & Y_l \ar[r]^{\psi_l} & Y_{l-1} \ar[r]^{\psi_{l-1}} & \cdots \ar[r]^{\psi_2} & Y_1 \ar[r]^{\psi_1} & Y_0 \coloneqq X,
}
\]
where $\psi_j \colon Y_j \to Y_{j-1}$ is an $H_{j-1}/H_j$-Galois cover of smooth projective connected $k$-curves for $1 \leq j \leq l$. Also note that $l$ is the minimal non-negative integer such that $H_{l} = \langle P_i | 1 \leq i \leq r \rangle$.

We ask whether the converse is also true.

\begin{que}\label{que_most_gen}
($Q[r,X,B,G]$) Let $r \geq 1$, $X$ be a smooth projective connected $k$-curve, $G$ be a finite group. Let $B = \{x_1, \cdots, x_r\} \subset X$ be a set of closed points in $X$. For $1 \leq i \leq r$ let $I_i$ be a subgroup of $G$ which is an extension of the $p$-group $P_i$ (possibly trivial) by a cyclic group of order prime-to-$p$ such that there is a tower
\[
\xymatrix{
Y_l \ar[r]^{\psi_l} & Y_{l-1} \ar[r]^{\psi_{l-1}} & \cdots \ar[r]^{\psi_2} & Y_1 \ar[r]^{\psi_1} & Y_0 \coloneqq X
}
\]
of covers of smooth projective connected $k$-curves where $\psi_j \colon Y_j \to Y_{j-1}$ is an $H_{j-1}/H_j$-Galois cover, and if $y_{i,j}$ is a point of $Y_j$ lying above $x_i$, then $I_i \cap H_{j-1}/I_i \cap H_j$ occurs as an inertia group above $y_{i,j-1}$ in the cover $\psi_j$, $1 \leq i \leq r$ (where $H_j$'s are as in Notation \ref{not_general_set_up}). Let $\psi \colon Y_l \to X$ denotes the composite morphism.

Whether there is a connected $G$-Galois cover $\phi$ of $X$ \'{e}tale away from $B$ dominating the cover $\psi$ such that $I_i$ occurs as an inertia group above the point $x_i$ for $1 \leq i \leq r$ ?
\end{que}

Since $p(I_i) = P_i \subset H_l$ for all $i$, the cover $\psi$ and each of the Galois covers $\psi_l$ are tamely ramified. If $B_{j-1} \subset Y_{j-1}$ denotes the brunch locus of the cover $\psi_{j}$, $1 \leq j \leq l$, then $H_{j-1}/H_j \in \pi_A^t(Y_{j-1} - B_{j-1})$. But only assuming this we get a negative answer to the above question as seen from the following example and so we need the hypothesis about the existence of the cover $\psi$ with the given ramification behavior.

\begin{example}\label{eg}
Take $r=1$, $X$ as an elliptic curve with origin $0$, $B = \{0\}$, $P_1$ as the trivial group, $G = I_1 = \bb{Z}/m$ where $m$ is coprime to $p$. Then $H_1$ is the trivial group. Since there is a connected $\bb{Z}/m$-Galois \'{e}tale cover of $X$, we have $G/H = \bb{Z}/m \in \pi_A^t(X-B)$. From the Riemann-Hurwitz formula we have
$$2g_Y -2 = m(2-2) + \frac{m}{m}(m-1).$$
So $m$ must be an odd integer. Thus if $m$ is an even integer, there is no $\bb{Z}/m$-Galois connected cover $Y \to X$ \'{e}tale away from $\{0\}$ over which it is totally ramified.
\end{example}

When all the $p$-groups $P_i$ are trivial, taking $\phi = \psi$ gives the answer to the question.

\textit{Now onward we will assume that $P_1$ is non-trivial.}

So $H_{j}$ is a non-trivial normal quasi $p$-subgroup of $H_{j-1}$ for $1 \leq j \leq l$. We make the following observation.

\begin{remark}\label{rmk_gen_que}
Let $X = \bb{P}^1$. Since there is no non-trivial \'{e}tale cover of $\bb{P}^1$, if each $I_i \subset H_l$, then $\psi$ is the identity map $\bb{P}^1 \to \bb{P}^1$. So $l=0$, i.e. $G = H_l = \langle P_i^G | 1\leq i \leq r \rangle$. This holds in particular, if all the $I_i$'s are $p$-groups. Also note that if $l=1$ and $\psi_1$ is a two point branched Galois cover of $\bb{P}^1$ (so in particular, $r = 2$),  by the Riemann Hurwitz formula, $\psi$ is the $\bb{Z}/n$-Galois Kummer cover of $\bb{P}^1$ branched at $\{x_1,x_2\}$ which is totally ramified above these points.
\end{remark}

With the above observation, a special case of Question \ref{que_most_gen} (i.e. $Q[r,\bb{P}^1,B,G]$ with each $I_i \subset H_l$) is the following which we pose as the Generalized Inertia Conjecture (GIC).

\begin{conj}\label{conj_proj_gen}
(GIC) Let $r \geq 1$ be an integer, $G$ be a finite quasi $p$-group. For $1 \leq i \leq r$ let $I_i$ be a subgroup of $G$ which is an extension of a $p$-group $P_i$ by a cyclic group of order prime-to-$p$ such that $G = \langle P_1^G, \cdots, P_r^G \rangle$. Let $B = \{x_1 ,\cdots, x_r\}$ be a set of closed points in $\bb{P}^1$. Then there is a connected $G$-Galois cover of $\bb{P}^1$ \'{e}tale away from $B$ such that $I_i$ occurs as an inertia group above the point $x_i$ for $1 \leq i \leq r$.
\end{conj}

\begin{remark}\label{rmk_1}
As in the case of Question \ref{que_most_gen} we again allow $P_i$'s to be trivial for $2 \leq i \leq r$.  Note that $r=1$ is the unsolved direction of the IC (Conjecture \ref{IC}).
\end{remark}

\begin{remark}
Under the notation and the hypothesis of Question \ref{que_most_gen} we can ask the following weaker question. Whether there is a connected $G$-Galois cover $\phi$ of $X$ \'{e}tale away from $r$ points such that $I_i$ occurs as an inertia group above these points ? But since the \'{e}tale fundamental group depends on the position of the removed points in general, we fix the set $B$ beforehand.
\end{remark}

Later we will see some partial answer to Question \ref{que_most_gen} for $l \in \{0,1\}$ as the application of the Formal Patching technique and by constructing covers given by the explicit equations. When the inertia groups $I_i$ are $p$-groups $P_i$ we have the following special case of the GIC which we see as a generalization of the PWIC (Conjecture \ref{PWIC}).

\begin{conj}\label{conj_GPWIC}
(Generalized Purely Wild Inertia Conjecture) Let $G$ be a finite quasi $p$-group. Let $P_1$, $\cdots$, $P_r$ be non-trivial $p$-subgroup of $G$ for some $r \geq 1$ such that $G = \langle P_1^G, \cdots, P_r^G \rangle$. Let $B = \{x_1 ,\cdots, x_r\}$ be a set of closed points in $\bb{P}^1$. Then there is a connected $G$-Galois cover of $\bb{P}^1$ \'{e}tale away from $B$ such that $P_i$ occurs as an inertia group above the point $x_i$ for $1 \leq i \leq r$.
\end{conj}

In the next section we see that for the groups $G$ for which the PWIC is known to be true, the Generalized Purely Wild Inertia Conjecture (GPWIC for short) is also true.

\section{Towards the GPWIC}\label{sec_GPWIC}
In this section we see some affirmative results for the Generalized Purely Wild Inertia Conjecture (Conjecture \ref{conj_GPWIC}). Let $G$ be a quasi $p$-group, $P_1$, $\cdots$, $P_r$ be non-trivial $p$-subgroups of $G$ for some $r \geq 1$ such that $G = \langle P_1^G, \cdots, P_r^G \rangle$. Let $B = \{x_1,\cdots, x_r\}$ be a set of closed points in $\bb{P}^1$. Then the GPWIC says that there is a connected $G$-Galois cover of $\bb{P}^1$ \'{e}tale away from $B$ such that $P_i$ occurs as an inertia group above the point $x_i$. We use these notation throughout this section and use the results from Section \cref{sec_fp} for the proofs. We start with the following group theoretic observation.

\begin{lemma}\label{lem_p-grp}
Let $G$ be a $p$-group, $r \geq 1$ be an integer. Let $P_1$, $\cdots$, $P_r$ be $p$-subgroups of $G$ such that $G =\langle P_1^G , \cdots, P_r^G \rangle$. Then $G =\langle P_1, \cdots, P_r \rangle$.
\end{lemma}

\begin{proof}
The result holds for an abelian $p$-group $G$. So assume that $G$ is non-abelian. Then the Frattini subgroup $\Phi(G)$ of $G$ is non-trivial. Consider the Frattini quotient $G \surj G/\Phi(G)$. Under this map $P_i$ has image $P_i/P_i \cap \Phi(G)$ and $G/\Phi(G) = \langle (P_i/P_i \cap \Phi(G))^{G/\Phi(G)} | 1 \leq i \leq r \rangle$. Since $G/\Phi(G)$ is elementary abelian, $G/\Phi(G) = \langle P_1/P_1 \cap \Phi(G), \cdots, P_r/P_r \cap \Phi(G) \rangle$. Since $\Phi(G)$ is the set of non-generators of $G$, we obtain
$$G = \langle P_1 ,\cdots, P_r, \Phi(G) \rangle = \langle P_1 ,\cdots, P_r \rangle.$$
\end{proof}

\begin{theorem}\label{thm_p-grp}
The GPWIC holds for every $p$-group.
\end{theorem}

\begin{proof}
By Lemma \ref{lem_p-grp}, $G =\langle P_1, \cdots, P_r \rangle$. We proceed via induction on $r$. The pair $(P_r,P_r)$ is realizable. By the induction hypothesis there is a connected $G_1 \coloneqq \langle P_1, \cdots, P_{r-1} \rangle$-Galois cover of $\bb{P}^1$ \'{e}tale away from $\{x_1,\cdots,x_{r-1}\}$ such that $P_i$ occurs as an inertia group above the point $x_i$ for $1 \leq i \leq r-1$. Now the result follows from Theorem \ref{thm_patching_covers}.
\end{proof}

\begin{theorem}\label{thm_strictly_div}
The GPWIC holds for any quasi $p$-group $G$ whose order is strictly divisible by $p$. In particular, it holds for Alternating groups $A_d$ with $p \leq d \leq 2p-1$ and for $PSL_2(p)$ when $p$ is an odd prime $\geq 5$.
\end{theorem}

\begin{proof}
Since $p^2 \nmid |G|$, each $P_i$ is a $p$-cyclic Sylow $p$-subgroup of $G$. Since $G$ is a quasi $p$-group, for each $i$, $G = \langle P_i^G \rangle$. We proceed by induction on $r$. By Raynaud's proof of the Abhyankar's Conjecture on the affine line (\cite[Corollary 2.2.2]{3}), the pair $(G,P_r)$ is realizable. Now we argue as in the proof of Theorem \ref{thm_p-grp}.
\end{proof}

\begin{theorem}\label{thm_GPWIC_pdt_small_order}
Let $u \geq 1$ be an integer. Let $G = G_1 \times \cdots \times G_u$ where each $G_i$ is either a non-abelian simple quasi $p$-group of order strictly divisible by $p$ or a simple Alternating group of degree coprime to $p$. Then the GPWIC is true for $G$.
\end{theorem}

\begin{proof}
For any subset $\Lambda \subseteq \{1, \cdots, u\}$, set $H_\Lambda \coloneqq \Pi_{\lambda \in \Lambda} G_\lambda$ and let $\pi_\Lambda \colon G \to H_\Lambda$ be the projection. For each $1 \leq i \leq r$ consider the set $\alpha_i$ consisting of $j \in \{1,\cdots,u\}$ such that $\pi_j(P_i)$ is non-trivial. Since $G_i$'s are simple non-abelian groups, the conjugates of $\pi_{\alpha_i}(P_i)$ in $H_{\alpha_i}$ generate $H_{\alpha_i}$. We proceed via induction on $r$. If $r=1$, by \cite[Remark 5.8, Corollary 5.4]{DK}, the pair $(G,P_1)$ is realizable. So let $r \geq 2$. By the induction hypothesis there is a connected $H_{\cup_{i=1}^{r-1} \alpha_i}$-Galois cover of $\bb{P}^1$ \'{e}tale away from $\{x_1,\cdots,x_{r-1}\}$ such that $P_i$ occurs as an inertia group above the point $x_i$ for $1 \leq i \leq r-1$. If $H_{\cup_{i=1}^{r-1} \alpha_i} = G$, the result follows by \cite[Theorem 2]{2}. Otherwise, set $S \coloneqq \{1,\cdots,u\} \setminus \cup_{i=1}^{r-1} \alpha_i$. Since $\cup_{1 \leq i \leq r} \alpha_i = \{ 1,\cdots, u \}$, we must have $S \subset \alpha_r$. Again by \cite[Remark 5.8, Corollary 5.7]{DK}, the pair $(H_{\alpha_r},P_r)$ is realizable. Then the result follows by applying Theorem \ref{thm_patching_covers} with the groups $H_{\cup_{i=1}^{r-1} \alpha_i}$ and $H_{\alpha_r}$.
\end{proof}

In the following we show that the GPWIC is true for certain product of groups if it is true for the individual groups.

\begin{theorem}\label{thm_pdt_arbit_groups}
Let $G_1$ and $G_2$ be two finite quasi $p$-groups such that they have no non-trivial quotient in common. If the GPWIC is true for the groups $G_1$ and $G_2$, then it is also true for $G_1 \times G_2$.
\end{theorem}

\begin{proof}
Set $G \coloneqq G_1 \times G_2$ and let $\pi_j \colon G \surj G_j$ denote the projections for $j \in \{1,2\}$. Let $r \geq 1$, $P_1$, $\cdots$, $P_r$ be non-trivial $p$-subgroups of $G$ such that $G = \langle P_1^G, \cdots, P_r^G \rangle$. Then for each $1 \leq i \leq r$, both of the groups $\pi_1(P_i)$ and $\pi_2(P_i)$ cannot be trivial. For $1 \leq i \leq r$, by the Goursat's lemma, there is $p$-group $Q_i$ (possibly trivial), normal subgroups $N_{j,i}$ of $\pi_j(P_i)$ such that $\pi_j(P_i)/N_{j,i} \cong Q_i$ and $P_i \cong \pi_1(P_i) \times_{Q_i} \pi_2(P_i)$. Let $Q'_i$ be a common quotient of $\pi_1(P_i)$ and $\pi_2(P_i)$ such that $Q_i$ is a quotient of $Q'_i$. Then $\pi_1(P_i) \times_{Q'_i} \pi_2(P_i) \subset \pi_1(P_i) \times_{Q_i} \pi_2(P_i)$ and so by \cite[Theorem 2]{2}, it is enough to consider $Q_i$ to be a maximal common quotient of $\pi_1(P_i)$ and $\pi_2(P_i)$. Now $G_j = \langle \pi_j(P_1)^{G_j}, \cdots, \pi_j(P_r)^{G_j} \rangle$. Let $B = \{x_1, \cdots, x_r\}$ be a sets of closed points in $\bb{P}^1$. By the hypothesis, for $j=1$, $2$, there is a connected $G_j$-Galois cover $f_j \colon Y_j \to \bb{P}^1$ \'{e}tale away from  $B$ such that $\pi_j(P_i)$ occurs as an inertia group above the point $x_i \in B$. By Lemma \ref{lem_common_cover_for_product}, for $j=1$ and $2$, there is a connected $G_j$-Galois cover $Z_j \to \bb{P}^1$ \'{e}tale away from $B$ such that $\pi_j(P_i)$ occurs as an inertia group above $x_i$, $1 \leq i \leq r$, and such that there is a point $z_{j,i} \in Z_j$ over $x_i$ with $K_{Z_1,z_{1,i}}^{N_{1,i}}/K_{\bb{P}^1,x_i} \cong K_{Z_2,z_{2,i}}^{N_{2,i}}/K_{\bb{P}^1,x_i}$ as $Q_i$-Galois extensions. Since $Q_i$ is a maximal common quotient of $\pi_1(P_i)$ and $\pi_2(P_i)$, the extensions $K_{Z_1,z_{1,i}}/K_{\bb{P}^1,x_i}$ and $K_{Z_2,z_{2,i}}/K_{\bb{P}^1,x_i}$ are linearly disjoint over $K_{Z_1,z_{1,i}}^{N_{1,i}} \cong K_{Z_2,z_{2,i}}^{N_{2,i}} /K_{\bb{P}^1,x_i}$. Let $W$ be a dominant connected component of the normalization of $Z_1 \times_{\bb{P}^1} Z_2$. Since $G_1$ and $G_2$ have no common non-trivial quotient in common, the cover $\Psi \colon W \to \bb{P}^1$ is a connected $G_1 \times G_2$-Galois cover. For a point $w = (z'_1, z'_2)$ in $W$ with $f_j(z'_j)=x$ the extension $K_{W,w}/K_{\bb{P}^1,x}$ is the compositum of the extension $K_{Z_1,z'_1}/K_{\bb{P}^1,x}$ with the extension $K_{Z_2,z'_2}/K_{\bb{P}^1,x}$. So the cover $\Psi$ is \'{e}tale away from $B$ and $\pi_1(P_i) \times_{Q_i} \pi_2(P_2) = P_i$ occurs as an inertia group above $x_i$.
\end{proof}

\begin{remark}\label{rmk_pdt_with_p-grp}
The above theorem generalizes \cite[Corollary 4.6]{15} where the result was proved for a perfect group $G_1$ and a $p$-group $G_2$.
\end{remark}

In the following we summarize the results of this section.

\begin{corollary}\label{cor_all}
The GPWIC (Conjecture \ref{conj_GPWIC}) is true for the following quasi $p$-groups $G$.
\begin{enumerate}
\item $G$ is a $p$-group;
\item $G$ has order strictly divisible by $p$;
\item $G = G_1 \times \cdots \times G_u$ where each $G_i$ is either a simple Alternating group of degree $d \geq p$, where $d=p$ or $(d,p)=1$ or $PSL_2(p)$ or a $p$-group or a simple non-abelian group of order strictly divisible by $p$.
\end{enumerate}
\end{corollary}

\section{Weaker results to the GPWIC}\label{sec_GPWIC_weak}
In this section we will see some weaker results to the GPWIC. Namely, when we allow the branch locus sufficiently large or if we allow bigger inertia groups, there are suitable covers with the prescribed ramification. We also show that when the group is $A_d$ it is enough to add only one more branch point. Using \cite[Theorem 2]{2} one can increase the wild part of the inertia groups of a cover. In particular, let $G$ be a quasi $p$-group, $P_1$, $\cdots$, $P_r$ be $p$-subgroups of $G$. Let $Q_i$ be a Sylow $p$-subgroup of $G$ containing $P_i$. Let $B = \{x_1,\cdots,x_r\}$ be a set of closed points of $\bb{P}^1$. Then there is a connected $G$-Galois cover of $\bb{P}^1$ \'{e}tale away from $B$ such that $Q_i$ occurs as an inertia group above $x_i$ for $1 \leq i \leq r$. In the following we show that the inertia groups can be taken as the Sylow $p$-subgroups of the normal quasi $p$-groups $\langle P_i^G \rangle$.

\begin{proposition}\label{prop_weaker_inertia}
Under the notation and hypothesis of Conjecture \ref{conj_GPWIC}, for each $i$, there is a $p$-subgroup $Q_i \supset P_i$ in $\langle P_i^G \rangle$ such that there is a connected $G$-Galois cover of $\bb{P}^1$ \'{e}tale away from $B$ and $Q_i$ occurs as an inertia group above the point $x_i$ for $1 \leq i \leq r$.
\end{proposition}

\begin{proof}
We proceed by induction on $r$. When $r=1$, it is the consequence of \cite[Theorem 2]{2}. So let $r \geq 2$. For $1 \leq i \leq r$, let $Q_i$ be a Sylow $p$-subgroup of $H_i = \langle P_i^G \rangle$ containing $P_i$. Since $H_i$ is a quasi $p$-group, $H_i = \langle Q_i^{H_i} \rangle$ and by the $r=1$ case, the pair $(H_i,Q_i)$ is realizable. By the induction hypothesis there is a connected $G_1 \coloneqq \langle H_1, \cdots, H_{r-1} \rangle$-Galois cover of $\bb{P}^1$ \'{e}tale away from $\{x_1,\cdots,x_{r-1}\}$ and $Q_i$ occurs as an inertia group above $x_i$. If $G_1 = G$, apply \cite[Theorem 2]{2}. Otherwise the result follows by Theorem \ref{thm_patching_covers} with $G_2 = H_r$.
\end{proof}

\begin{proposition}\label{prop_weaker_branch_locus}
Assume that the hypothesis and notation of Conjecture \ref{conj_GPWIC} hold. For each $1 \leq i \leq r$, let $C_i \coloneqq \{P_{i,j}\}_{1 \leq j \leq t_i}$ be the set of all conjugates of $P_i$ in $G$. Let $\emptyset \neq A_i \subset C_i$ such that $G = \langle P | P \in A_i, 1 \leq i \leq r \rangle$. Let $l \coloneqq \Sigma_i |A_i|$ and let $B = \{x_{i,j}| 1 \leq i \leq r, P_{i,j} \in A_i\}$ be a set of $l$ closed points in $\bb{P}^1$ such that $x_{i,1} = x_i$ for each $i$. Then there is a connected $G$-Galois cover of $\bb{P}^1$ \'{e}tale away from $B$ and $P_i$ occurs as an inertia group above the point $x_{i,j} \in B$.
\end{proposition}

\begin{proof}
Since the inertia groups above a points in a connected Galois cover are conjugates, it is enough to prove that there is a connected $G$-Galois cover of $\bb{P}^1$ \'{e}tale away from $B$ such that $P_{i,j}$ occurs as an inertia group above the point $x_{i,j} \in B$. We proceed via induction on $l$. If $l=1$ then $G = P_1$ and the result follows. Let $l \geq 2$. Fix an $i_0$, $1 \leq i_0 \leq r$, and element $P_{i_0,j_0}$ in $A_{i_0}$. Then the pair $(P_{i_0,j_0},P_{i_0,j_0})$ is realizable. By the induction hypothesis, we may assume that there is a connected $G_1 \coloneqq \langle \{P_{i,j}| 1 \leq i \leq r, j \in A_i\} - \{P_{i_0,j_0}\} \rangle$-Galois cover of $\bb{P}^1$ \'{e}tale away from $B - \{x_{i_0,j_0}\}$ and $P_{i,j}$ occurs as an inertia group above the point $x_{i,j} \in B$ for $(i,j) \neq (i_0,j_0)$. Now use Theorem \ref{thm_patching_covers} with $G_2 = P_{i_0,j_0}$.
\end{proof}

By the above proposition , if we allow enough number of branch points we can obtain covers with the desired purely wild ramification. By \cite[Corollary 5.5]{DK}, for $d=p$ or $d$ is coprime to $p$, the GPWIC holds for $A_d$. So assume that $a \geq 2$ and $d = ap$. The following result shows that in this case we only need one extra branched point.

\begin{proposition}\label{prop_A_d_pdt_cycle}
Let $p$ be an odd prime, $a \geq 2$ be an integer, $d = a p$. Let $r \geq 1$ be an integer and $P_1$, $\cdots$, $P_r$ be non-trivial $p$-subgroups of $A_d$. Let $B = \{x_1,\cdots,x_r\}$ be a set of closed points in $\bb{P}^1$ and let $x_0 \in \bb{P}^1$ be a closed point outside $B$. Fix $1 \leq i_0 \leq r$. Then there is a connected $A_d$-Galois cover of $\bb{P}^1$ \'{e}tale away from $B \sqcup \{x_0\}$ such that $P_i$ occurs as an inertia group above $x_i$ and $P_{i_0}$ occurs as an inertia group above $x_0$.
\end{proposition}

\begin{proof}
We may assume that $i_0 = 1$. By \cite[Theorem 2]{2} it is enough to consider the case $r=1$ and when $P_1 = \langle \tau \rangle$ for an element $\tau$ of order $p$. Without lose of generality we may assume that $\tau = \tau_1 \cdots \tau_v$ where $\tau_i$ is a the $p$-cycle $((i-1)p+1,\cdots,ip)$, $1 \leq i \leq v$. By \cite[Corollary 5.6]{DK} we may assume that $v = a$.

For $1 \leq i \leq a$, set $H_{i1}\coloneqq \tx{Alt}(\{(i-1)p+1,\cdots,ip\})$. For $1 \leq j \leq a-1$ set $H_{j2} \coloneqq \tx{Alt}(\{(j-1)p+2, \cdots, jp+1\})$ and let $H_{a2} \coloneqq \tx{Alt}(\{(a-1)p+2,\cdots,ap,1\})$. For $i=1$, $2$, set $G_i \coloneqq H_{1i} \times \cdots H_{ai} \subset A_d$. For $1 \leq j \leq a$, consider the $p$-cycle $\sigma_j$ given by $\sigma_j \coloneqq ((j-1)p+2,\cdots,jp+1)$ for $1 \leq j \leq a-1$, $\sigma_a \coloneqq ((a-1)p+2, \cdots, ap, 1)$. Consider the element $\sigma \coloneqq \sigma_1 \cdots \sigma_a$ in $A_d$ of order $p$. Now by Theorem \ref{thm_GPWIC_pdt_small_order} the pairs $(G_1, \langle (\tau_1,\cdots, \tau_a) \rangle)$ and $(G_2, \langle (\sigma_1,\cdots,\sigma_a) \rangle)$ are realizable. Set $G \coloneqq \langle G_1, G_2 \rangle \subset A_d$. Since each $H_{ij} $ are generated by $p$-cycles, so is $G$. Also the $3$-cycle $(1,2,3) \in H_{11}$ is contained in $G$. So by \cite[Lemma 4.4.4]{TGT} $G = A_d$. Since $\sigma$ is a conjugate of $\tau$ in $A_d$, by Theorem \ref{thm_patching_covers}, there is a connected $A_d$-Galois cover of $\bb{P}^1$ \'{e}tale away from $\{x_0,x_1\}$ such that $\langle \tau \rangle$ occurs as an inertia group over $x_0$ and over $x_1$.
\end{proof}

\section{Towards the general question}\label{sec_que_evidence}
In this section we see some evidence for $Q[r,X,B,G]$ (Question \ref{que_most_gen}). The following result is a consequence of the Formal Patching technique (Theorem \ref{thm_patching_covers}).

\begin{proposition}\label{prop_part_most_gen}
Let $r \geq 1$, $X$ be a smooth projective connected $k$-curve, $G$ be a finite group. Let $B = \{x_1, \cdots, x_r\} \subset X$ be a set of closed points in $X$. For $1\leq i \leq r$, let $I_i$ be a subgroup of $G$ which is an extension of a $p$-group $P_i$ by a cyclic group of order prime-to-$p$, and set $H \coloneqq \langle P_i^G | 1 \leq i \leq r \rangle$. Assume that $H$ has a complement $H'$ in $G$. Suppose that the following hold.
\begin{enumerate}
\item There is a connected $H'$-Galois \'{e}tale cover $\psi \colon Z \to X$.
\item There is a connected $H$-Galois cover of $\bb{P}^1$ \'{e}tale away from a set $\{\eta_1, \cdots, \eta_r\}$ of $r$-distinct points such that $I_i$ occurs as an inertia group above the point $\eta_i$ for $1 \leq i \leq r$.
\end{enumerate}
Then there is a connected $G$-Galois cover of $X$ \'{e}tale away from a set $B' = \{x'_1,\cdots,x'_r\}$ of closed points such that $I_i$ occurs as an inertia group above the point $x'_i$, $1 \leq i \leq r$. Moreover, we can choose an $i$ such that $x_i' = x_i$. Furthermore, if each $(H,I_i)$ is realizable, we can take $x'_i = x_i$ for all $1 \leq i \leq r$.
\end{proposition}

\begin{proof}
Since $\psi$ is an unramified cover, each $I_i \subset H$. By the hypothesis, $G/H$ is a subgroup of $G$ which together with $H$ generates $G$. By Theorem \ref{thm_patching_covers}, there is a connected $G$-Galois cover of $X$ \'{e}tale away from $B'$ such that $I_i$ occurs as an inertia group above the point $x'_i$ for $1 \leq i \leq r$ and we can choose one $1 \leq i \leq r$ such that $x_i' = x_i$. For the last assertion we use Theorem \ref{thm_patching_covers} inductively for $r$.
\end{proof}

\begin{remark}
Note that when $(|G/H|,p) =1$, by the Schur-Zassenhaus Theorem, the group $H$ always has a complement in $G$.
\end{remark}

The hypotheses of the above result assumes the existence of the unramified cover $\psi$, which are not well understood in general. But they are known when the Galois group either has order prime-to-$p$ or when the Galois group is a $p$-group. Using these structures we have the following result.

\begin{corollary}\label{cor_most_gen}
Under the hypotheses of Question \ref{que_most_gen} suppose that one of the following holds.
\begin{enumerate}
\item $(|G/H|,p)=1$ and the cover $\psi$ is an \'{e}tale $G/H$-Galois cover of $X$;
\item $G = H \rtimes H'$ for some $p$-group $H'$ of rank $s$ and $X$ has $p$-rank $\geq s$ and $\psi$ is an \'{e}tale $H'$-Galois cover of $X$.
\end{enumerate}
Assume that each $I_i \subset \langle P_i^G \rangle$ and each pair $(\langle P_i^G \rangle,I_i)$ is realizable. Then for any set $B =\{x_1, \cdots, x_r\}$ of closed points in $X$, there is a connected $G$-Galois cover of $X$ \'{e}tale away from $B$ such that $I_i$ occurs as an inertia group above the point $x_i$ for $1 \leq i \leq r$.
\end{corollary}

\begin{remark}
Note that the above results can be applied to the Question $Q[r,X,B,G]$ if there exists $0 \leq j \leq l$ such that $H_j$ is normal in $G$ which has a complement in $G$, the composite cover $\psi_j \circ \cdots \circ \psi_1$ is \'{e}tale and such that each pair $(\langle P_i^G \rangle,I_i)$ is realizable.
\end{remark}

Using the IC for the Alternating groups proved in Section \cref{sec_alt_IC} the above Corollary implies the following result.

\begin{corollary}\label{cor_gen_que_sym}
($Q[1,X,\{*\},S_d]$) Let $p$ be an odd prime and $X$ be any smooth projective $k$-curve of genus $\geq 1$. Then for $r = 1$, Question \ref{que_most_gen} has an affirmative answer for the group $S_p$ and when $p \equiv 2 \pmod{3}$ for the groups $S_{p+1}$, $S_{p+2}$, $S_{p+3}$, $S_{p+4}$.
\end{corollary}

\begin{proof}
Let $d = p$ or when $p \equiv 2 \pmod{3}$, $d \in \{p+1,p+2,p+3,p+4\}$. Set $G = S_d$. Let $x \in X$ be a closed point, $I \subset G$ be an extension of a $p$-group $P$ by a cyclic group of order prime-to-$p$. Then $\langle P^G \rangle = A_d$. Let $\psi \colon Y \to X$ be a connected $\bb{Z}/2$-Galois cover of $X$ \'{e}tale away from $x$. By the Riemann-Hurwitz formula, the ramification index above $x$ must be an odd integer. So $\psi$ is \'{e}tale everywhere and $I \subset A_d$. Now the result follows from Corollary \ref{cor_most_gen} applied to \cite[Theorem 1.2]{BP}, \cite[Theorem 1.2]{8} and Theorem \ref{thm_A_p+1}--Theorem \ref{thm_A_p+4}.
\end{proof}

Again using the IC for the Alternating groups proved in Section \cref{sec_alt_IC} together with Formal Patching technique we have the following result towards the GIC (Conjecture \ref{conj_proj_gen}).

\begin{proposition}\label{prop_GIC_A_d_small}
Let $p \geq 5$ be a prime number. Let $d = p$ or when $p \equiv 2 \pmod{3}$, $d \in \{p+1,p+2,p+3,p+4\}$. Let $r \geq 1$ be an integer. For $1 \leq i \leq r$, let $I_i$ be a subgroup of $A_d$ which is an extension of a $p$-group $P_i$ (possibly empty) by a cyclic group of order prime-to-$p$ such that $P_1$ is non-trivial, and if $P_i$ is trivial for some $i$, there is a $2 \leq j \leq r$, $j \neq i$, such that $I_i = I_j$. Then there exists a connected $A_d$-Galois cover of $\bb{P}^1$ \'{e}tale away from a set $B = \{x_1,\cdots,x_r\}$ of closed points in $\bb{P}^1$ such that $I_i$ occurs as an inertia group above $x_i$, $1 \leq i \leq r$. Moreover, if all the $I_i$ are equal whenever $P_i$ is trivial, the set $B$ can be chosen arbitrarily.
\end{proposition}

\begin{proof}
For $P_i = \{1\}$ set $A_i \coloneqq \{2 \leq j \leq r| I_i = I_j\}$. For $i$ such that $P_i = \{1\}$, let $\beta$ be a generator of $I_i$ and by the Kummer theory there is a connected $\langle \beta \rangle$-Galois cover of $\bb{P}^1$ \'{e}tale away from $A_i$ which is totally ramified over each $x_j$, $j \in A_i$. Now the result follows by inductively applying Theorem \ref{thm_patching_covers} to the above covers and the covers obtained from \cite[Theorem 1.2]{BP}, \cite[Theorem 1.2]{8} and Theorem \ref{thm_A_p+1}--Theorem \ref{thm_A_p+4}.
\end{proof}

Now onward we study the general question for $X = \bb{P}^1$ and $r=2$. We have the following result when when $G = P \rtimes \bb{Z}/n$ for a $p$-group $P$ and $n$ coprime to $p$.

\begin{theorem}\label{thm_semidirect}
($Q[2,\bb{P}^1,\{0,\infty\},P \rtimes \bb{Z}/n]$) Let $G = P \rtimes \bb{Z}/n$ for a $p$-group $P$ and $(p,n)=1$. Then Question \ref{que_most_gen} has an affirmative answer for $G$, $\bb{P}^1$ and $r=2$.
\end{theorem}

\begin{proof}
In view of Theorem \ref{thm_p-grp} we may assume $n \geq 2$. Let $P_1$, $P_2$ be two $p$-subgroups of $P$ where $P_1$ is non-trivial and $P_2$ is possible trivial and such that $\langle P_1^G, P_2^G \rangle = P$. For $i=1$, $2$, let $I_i = P_i \rtimes \bb{Z}/m_i$. Let $\psi \colon Y \to \bb{P}^1$ be a connected $\bb{Z}/n$-Galois cover \'{e}tale away from $\{0,\infty\}$ such that $\bb{Z}/m_1$ occurs as an inertia group above $0$ and $\bb{Z}/m_2$ occurs as an inertia group above $\infty$. By the Riemann Hurwitz formula $m_1 = m_2 = n$, and $\psi$ is the $\bb{Z}/n$-Galois Kummer cover totally ramified over $0$ and $\infty$. Since $I_i$ normalizes $P_i$, $\bb{Z}/n$ also normalizes $P_i$ in $G$. So $\langle P_1^P, P_2^P \rangle = P$. By Lemma \ref{lem_p-grp}, $P = \langle P_1, P_2 \rangle$. Consider the connected $P_1 \rtimes \bb{Z}/n$-Galois HKG cover $\psi$ of $\bb{P}^1$ \'{e}tale away from $\{0,\infty\}$ which is totally ramified above $\infty$ and such that $\bb{Z}/n$ occurs as an inertia group above $0$ (Theorem \ref{thm_embedding}). If $P_2$ is the trivial group, $G = P_1 \rtimes \bb{Z}/n$. Otherwise apply \cite[Theorem 3.6]{AC_embed_Harbater} to this cover to obtain the result.
\end{proof}

For the rest of this article we study some $S_d$-Galois covers with $X=\bb{P}^1$ and $r=2$. These realization results are the evidences for the Question $Q[2,\bb{P}^1,\{0,\infty\},S_d]$. When $P_2$ is the trivial group, we have seen some of the cases that can occur from studying explicit equations (Section \cref{sec_explicit_eq}). The following result shows the existence of another such cover with $P_2 = \{1\}$ and with the same tame part of the inertia groups over both points as an application of Theorem \ref{thm_embedding} to \cite[Corollary 5.5]{DK}.

\begin{corollary}\label{cor_S_d_same_inertia}
Let $p$ be an odd prime, $d \geq p$ be an integer such that either $d=p$ or $(d,p)=1$. Let $I$ be subgroup of $S_d$ which is an extension of a $p$-subgroup $P$ by a cyclic group of order prime-to-$p$ whose generator is an odd permutation $\gamma$ in $S_d$. Then there is a connected $S_d$-Galois cover of $\bb{P}^1$ \'{e}tale away from $\{0,\infty\}$ such that $\langle \gamma \rangle$ occurs as an inertia group at a point over $0$ and $I$ occurs as an inertia group at a point over $\infty$.
\end{corollary}

\begin{proof}
Set $n \coloneqq \tx{ord}(\gamma)$. Consider the $\bb{Z}/n$-Galois Kummer cover $\psi \colon \bb{P}^1 \to \bb{P}^1$ totally ramified over $\{0,\infty\}$ and \'{e}tale everywhere else. By \cite[Corollary 5.5]{DK}, the pair $(A_d, P)$ is realizable. Now the result follows from Theorem \ref{thm_embedding} by taking $\Gamma = S_d = \langle A_d, \langle \gamma \rangle \rangle$.
\end{proof}

Now onward we consider the cases where $P_1$ and $P_2$ are both non-trivial subgroups of $S_d$, $p \leq d \leq 2p-1$. Without lose of generality, we may assume that $P_i = \langle \tau \rangle$ for $i=1$, $2$, where $\tau$ is the $p$-cycle $(1, \cdots, p)$. We prove that for $1 \leq j_i \leq p-1$, $\omega_i \in \tx{Sym}\{p+1,\cdots,d\}$ such that $\theta^{j_i} \omega_i$ is an odd permutation, there is a connected $S_d$-Galois cover of $\bb{P}^1$ \'{e}tale away from $\{0,\infty\}$ such that $\langle \tau \rangle \rtimes \langle \theta^{j_1} \omega_1 \rangle$ occurs as an inertia group above $0$ and $\langle \tau \rangle \rtimes \langle \theta^{j_2} \omega_2 \rangle$ occurs as an inertia group above $\infty$. Similar to the case of proving the IC for the Alternating groups, we make the following reduction steps.

\begin{remark}\label{rmk_reductions_Symm}
Since $\langle \theta^i \rangle = \langle \theta^{(i,p-1)} \rangle$, it is enough to consider $i|(p-1)$. Also since two elements in $S_d$ are conjugate if and only if they have the same cycle structure, by Theorem \ref{thm_patching_covers}, it is enough to show that there exists a $\gamma \in S_d$ so that the following holds. For each $1 \leq i \leq p-1$ dividing $p-1$ and $\omega \in \tx{Sym}\{p+1,\cdots,d\}$ with $\theta^i\omega$ an odd permutation, there is a connected $S_d$-Galois cover of $\bb{P}^1$ \'{e}tale away from $\{0,\infty\}$ such that $\langle \gamma \rangle$ occurs as an inertia group above $0$ and $\langle \tau \rangle \rtimes \langle \theta^i \omega \rangle$ occurs as an inertia group above $\infty$.
\end{remark}

\begin{theorem}\label{thm_S_p_both_Sylow_non-trivial}
Let $p \geq 5$ be a prime. Let $I_1 \coloneqq \langle \tau \rangle \rtimes \langle \theta^i \rangle$, $I_2 \coloneqq \langle \tau \rangle \rtimes \langle \theta^j \rangle$ be two subgroups of $S_p$ for some $1 \leq i, j \leq p-1$ odd integers. Then there is a connected $S_p$-Galois cover of $\bb{P}^1$ \'{e}tale away from $\{0,\infty\}$ such that $I_1$ occurs as an inertia group above $0$ and $I_2$ occurs as an inertia group above $\infty$.
\end{theorem}

\begin{proof}
By Remark \ref{rmk_reductions_Symm}, we need to show that for some odd permutation $\gamma \in S_p$, for each odd divisor $i$ of $p-1$, there is a connected $S_p$-Galois cover of $\bb{P}^1$ \'{e}tale away from $\{0,\infty\}$ such that $\langle \gamma \rangle$ occurs as an inertia group above $0$ and $\langle \tau \rangle \rtimes \langle \theta^i \omega \rangle$ occurs as an inertia group above $\infty$.

Consider the degree $p$ cover $\psi \colon Y \to \bb{P}^1$ given by the affine equation $f(x,y) = 0$ where
$$f(x,y) = y^p - y^2 -x = 0.$$
Let $\phi \colon \widetilde{Y} \to \bb{P}^1$ be its Galois closure. By Remark \ref{rmk_Abhyankar_deg_p}, $\phi$ is a connected $S_p$-Galois cover of $\bb{P}^1$ \'{e}tale away from $\{0,\infty\}$ such that the inertia groups over $0$ are $2$-cyclic groups generated by transpositions and $\langle \tau \rangle \rtimes \langle \theta \rangle$ occurs as an inertia group above $\infty$. Since $i$ is odd, after the $[i]$-Kummer pullback of $\phi$, we obtain a connected $S_p$-Galois cover of $\bb{P}^1$ \'{e}tale away from $\{0,\infty\}$ such that the inertia groups over $0$ are $2$-cyclic groups generated by transpositions and $\langle \tau \rangle \rtimes \langle \theta^i \rangle$ occurs as an inertia group above $\infty$.
\end{proof}

\begin{theorem}\label{thm_S_p+1_both_Sylow_non-trivial}
Let $p \equiv 2 \pmod{3}$ be an odd prime. Let $I_1 \coloneqq \langle \tau \rangle \rtimes \langle \theta^i \rangle$, $I_2 \coloneqq \langle \tau \rangle \rtimes \langle \theta^j \rangle$ be two subgroups of $S_{p+1}$ for some $1 \leq i, j \leq p-1$ odd integers. Then there is a connected $S_{p+1}$-Galois cover of $\bb{P}^1$ \'{e}tale away from $\{0,\infty\}$ such that $I_1$ occurs as an inertia group above $0$ and $I_2$ occurs as an inertia group above $\infty$.
\end{theorem}

\begin{proof}
Consider the $S_{p+1}$-Galois cover $\phi \colon \widetilde{Y} \to \bb{P}^1$ which is the Galois closure of the degree-$(p+1)$ cover of $\bb{P}^1$ given by the affine equation (\ref{eq_p+1}). Then as in the proof of Theorem \ref{thm_A_p+1} the cover $\phi$ is \'{e}tale away from $\{0, \infty\}$ such that $\langle (1, \cdots, p-2) (p-1, p) \rangle$ occurs as an inertia group above $0$ and $\langle \tau \rangle \rtimes \langle \theta \rangle$ occurs as an inertia group above $\infty$. After a $[p-2]$-Kummer pullback we obtain a connected $S_{p+2}$-Galois cover of $\bb{P}^1$ \'{e}tale away from $\{0, \infty\}$ such that $\langle (p-1, p) \rangle$ occurs as an inertia group above $0$ and $\langle \tau \rangle \rtimes \langle \theta \rangle$ occurs as an inertia group above $\infty$. Now we can argue as in the proof of Theorem \ref{thm_S_p_both_Sylow_non-trivial}.
\end{proof}

\begin{theorem}\label{thm_S_P+2_both_Sylow_non-trivial}
Let $p \equiv 11 \pmod{12}$ be a prime such that $p$ is not of the form $l+1$ for any prime $l \geq 5$. For $i=1$, $2$ let $I_i \coloneqq \langle \tau \rangle \rtimes \langle \theta^{j_i} \omega_i \rangle$ be a subgroup of $S_{p+2}$ for some $1 \leq j_i \leq p-1$ and $\omega_i \in \tx{Sym}\{p+1,p+2\}$ such that $\theta^{j_i}\omega_i$ is an odd permutation. Then there is a connected $S_{p+2}$-Galois cover of $\bb{P}^1$ \'{e}tale away from $\{0,\infty\}$ such that $I_1$ occurs as an inertia group above $0$ and $I_2$ occurs as an inertia group above $\infty$.
\end{theorem}

\begin{proof}
Let $p \equiv 11 \pmod{12}$ or equivalently, $p$ satisfies $p \equiv 2 \pmod{3}$ and $p \equiv 3 \pmod{4}$. Let $\gamma$ be the $(p+1)$-cycle $(1, \cdots, p+1)$ in $S_{p+2}$. By Remark \ref{rmk_reductions_Symm}, it is enough to show that the following cases hold.
\begin{enumerate}
\item For each odd $1 \leq i \leq p-1$ dividing $p-1$ there is a connected $S_{p+2}$-Galois cover of $\bb{P}^1$ \'{e}tale away from $\{0,\infty\}$ such that $\langle \gamma \rangle$ occurs as an inertia group above $0$ and $\langle \tau \rangle \rtimes \langle \theta^i \rangle$ occurs as an inertia group above $\infty$;
\item For each even $1 \leq j \leq p-1$ dividing $p-1$ there is a connected $S_{p+2}$-Galois cover of $\bb{P}^1$ \'{e}tale away from $\{0,\infty\}$ such that $\langle \gamma \rangle$ occurs as an inertia group above $0$ and $\langle \tau \rangle \rtimes \langle \theta^j (p+1, p+2) \rangle$ occurs as an inertia group above $\infty$.
\end{enumerate}

First set $s = 2 = r$, $n_1 = p+1$, $n_2 = 1$, $m_1 = 1$, $m_2 = 1$. Then by Lemma \ref{lem_Ass_holds}(3), Assumption \ref{ass_num} holds. Consider the Galois cover $\phi_1$ of $\bb{P}^1$ which is the Galois closure of the degree-$(p+2)$ cover of $\bb{P}^1$ given by the affine equation (\ref{eq_explicit_general_cover}). By Proposition \ref{prop_ded_d_cover}, $\phi_1$ is a connected Galois cover of $\bb{P}^1$ with group $G_1$, a primitive subgroup of $S_{p+2}$, which is \'{e}tale away from $\{0,\infty\}$ such that $\langle \gamma \rangle$ occurs as an inertia group above $0$ and since $p \equiv 2 \pmod{3}$, $ \langle \tau \rangle \rtimes \langle \theta \rangle$ occurs as an inertia group above $\infty$.

Now let $s = 2$, $r=1$, $n_1 = p+1$, $n_2 = 1$. Then by Lemma \ref{lem_Ass_holds}(1) Assumption \ref{ass_num} holds. Consider the Galois cover $\phi_2$ of $\bb{P}^1$ which is the Galois closure of the degree-$(p+2)$ cover of $\bb{P}^1$ given by the affine equation (\ref{eq_explicit_general_cover}). By Proposition \ref{prop_ded_d_cover}, $\phi_2$ is a connected Galois cover of $\bb{P}^1$ with group $G_2$, a primitive subgroup of $S_{p+2}$, which is \'{e}tale away from $\{0,\infty\}$ such that $\langle \gamma \rangle$ occurs as an inertia group above $0$ and $\langle \tau \rangle \rtimes \langle \theta^2 (p+1,p+2)\rangle$ occurs as an inertia group above $\infty$.

The Galois groups $G_1$ and $G_2$ are primitive subgroup of $S_{p+2}$ containing a $p$-cycle which fixes $2$ points in $\{1,\cdots,p+2\}$ and a $(p+1)$-cycle which fixes $1$ point. Since $p$ is not of the form $l+1$ for a prime $l \geq 5$, by [Theorem 1.2]\cite{Jones}, both $G_1$ and $G_2$ contain $A_{p+2}$. Since $\gamma$ is an odd permutation, $G_i = S_{p+2}$ for $i=1$, $2$. Since $i$ is an odd divisor of $p-1$, $(p+1,i)=1$. After an $[i]$-Kummer pullback of $\phi_1$ we obtain a connected $S_{p+2}$-Galois cover of $\bb{P}^1$ \'{e}tale away from $\{0,\infty\}$ such that $\langle \gamma \rangle$ occurs as an inertia group over $0$ and $\langle \tau \rangle \rtimes \langle \theta^i \rangle$ occurs as an inertia group above $\infty$. Since $4 \nmid p$, $j/2$ is odd. After a $[j/2]$-Kummer pullback of $\phi_2$ we obtain a connected $S_{p+2}$-Galois cover of $\bb{P}^1$ \'{e}tale away from $\{0,\infty\}$ such that $\langle \gamma \rangle$ occurs as an inertia group over $0$ and $\langle \tau \rangle \rtimes \langle \theta^j (p+1, p+2) \rangle$ occurs as an inertia group above $\infty$.
\end{proof}

\begin{theorem}\label{thm_S_P+3_both_Sylow_non-trivial}
Let $p$ be a prime such that $p \equiv 11 \pmod{12}$. For $i=1$, $2$ let $I_i \coloneqq \langle \tau \rangle \rtimes \langle \theta^{j_i} \omega_i \rangle$ be a subgroup of $S_{p+3}$ for some $1 \leq j_i \leq p-1$ and $\omega_i \in \tx{Sym}\{p+1,p+2,p+3\}$ such that $\theta^{j_i}\omega_i$ is an odd permutation. Then there is a connected $S_{p+3}$-Galois cover of $\bb{P}^1$ \'{e}tale away from $\{0,\infty\}$ such that $I_1$ occurs as an inertia group above $0$ and $I_2$ occurs as an inertia group above $\infty$.
\end{theorem}

\begin{proof}
Let $p \equiv 11 \pmod{12}$. Consider the $(p+3)$-cycle $\gamma \coloneqq (1, \cdots, p+3)$ in $S_{p+3}$. In view of Remark \ref{rmk_reductions_Symm} we show that there is a connected $S_{p+3}$-Galois cover of $\bb{P}^1$ \'{e}tale away from $\{0,\infty\}$ such that $\langle \gamma \rangle$ occurs as an inertia groups above $0$ and $I = \langle \tau \rangle \rtimes \langle \beta \rangle$ occurs as an inertia group above $\infty$ where $\beta$ is of the following: $\beta = \theta^{i_1}$ for all odd integer $i_1|(p-1)$, $\beta = \theta^{i_2} (p+1, p+2)$ for any even integer $i_2|(p-1)$ and $\beta = \theta^{i_3} (p+1, p+2, p+3)$ for all odd integer $i_3|(p-1)$.

First set $s=1$, $r=3$, $m_1 = m_2 = m_3 =1$. Choose an element $w \in k$ such that $w^2 = p-3$. Then with the choice $(\alpha_1,\beta_1,\beta_2,\beta_3) = (0,1, \frac{w-1}{2}, -\frac{w+1}{2})$, Assumption \ref{ass_num} is satisfied. By Proposition \ref{prop_ded_d_cover} there is a connected $S_{p+3}$-Galois cover of $\bb{P}^1$ \'{e}tale away from $\{0,\infty\}$ such that $\langle \gamma \rangle$ occurs as an inertia groups above $0$ and since $p \equiv 2 \pmod{3}$, $I = \langle \tau \rangle \rtimes \langle \theta \rangle$ occurs as an inertia group above $\infty$. Since $i_1 | (p-1)$ is odd and $(p-1,p+3)= (p-1,4)=2$, after an $[i]$-Kummer pullback we obtain the required cover with $\beta = \theta^{i_1}$.

Now take $s=1$, $r=2$, $m_1 = 2$, $m_2 = 1$. Then by Lemma \ref{lem_Ass_holds} for the choice $(\alpha_1,\beta_1,\beta_2) = (0,1,-2)$ Assumption \ref{ass_num} is satisfied. By Proposition \ref{prop_ded_d_cover} there is a connected $S_{p+3}$-Galois cover of $\bb{P}^1$ \'{e}tale away from $\{0,\infty\}$ such that $\langle \gamma \rangle$ occurs as an inertia groups above $0$ and $I = \langle \tau \rangle \rtimes \langle \theta^2 (p+1, p+2) \rangle$ occurs as an inertia group above $\infty$. As $p \equiv 3 \pmod{4}$ and $i_2 | (p-1)$ is even, $(p+3,i_2/2)=1$. After an $[i_2/2]$-Kummer pullback we obtain the required cover with $\beta = \theta^{i_2} (p+1, p+2)$.

Finally, take $s = 1 = r$. Then by Lemma \ref{lem_Ass_holds} Assumption \ref{ass_num} holds and by Proposition \ref{prop_ded_d_cover} there is a connected $S_{p+3}$-Galois cover of $\bb{P}^1$ \'{e}tale away from $\{0,\infty\}$ such that $\langle \gamma \rangle$ occurs as an inertia groups above $0$ and $I = \langle \tau \rangle \rtimes \langle \theta (p+1, p+2, p+3) \rangle$ occurs as an inertia group above $\infty$. Since $i_3$ is an odd divisor of $p-1$, via an $[i_3]$-Kummer pullback we obtain the required cover with $\beta = \theta^{i_3} (p+1, p+2, p+3)$.
\end{proof}

\bibliographystyle{amsplain}

\end{document}